\documentclass[amstex,12pt, amssymb]{article}

\usepackage{mathtext}
\usepackage[cp1251]{inputenc}
\usepackage[T2A]{fontenc}
\usepackage[dvips]{graphicx}
\usepackage{amsmath}
\usepackage{amssymb}
\usepackage{amsxtra}
\usepackage{latexsym}
\usepackage{ifthen}

\newcounter{lemma}[section]

\newcounter{corollary}[section]

\newcounter{remark}[section]

\newcounter{theorem}[section]

\newcounter{proposition}[section]

\newcounter{example}

\numberwithin{equation}{section}

\begin{document}

\markboth{E~.SEVOST'YANOV, V.~TARGONSKII}{\centerline{ON DIRECT AND
INVERSE POLETSKY INEQUALITY ...}}

\def\cc{\setcounter{equation}{0}
\setcounter{figure}{0}\setcounter{table}{0}}

\overfullrule=0pt


\author{EVGENY SEVOST'YANOV, VALERY TARGONSKII}

\title{
{\bf ON DIRECT AND INVERSE POLETSKY INEQUALITY WITH A TANGENTIAL
DILATATION ON THE PLANE}}

\date{\today}
\maketitle

\begin{abstract}
This article is devoted to the study of mappings defined in the
region on the plane. Under certain conditions, the upper estimate of
the distortion of the modulus of families of paths is obtained.
Similarly, the upper estimate of the modulus of the families of
paths in the pre-image under the mapping is also obtained.
\end{abstract}


\section{Introduction}

Some times ago we obtained various inequalities for the distortion
of the modulus of families of paths (see, e.g., \cite{SalSev$_1$},
\cite{Sev$_1$}, \cite{Sev$_3$} and \cite{SST}). As a rule, we deal
with the so-called outher or inner dilatations, the use of which is
generally accepted (see, e.g., \cite[Theorem~34.4]{Va},
\cite[Theorems~8.5--8.6]{MRSY}). We note that there are also some
other characteristics that describe the distortion of the module
under mappings, such as, for example, tangential dilatations (see,
e.g., \cite[Theorem~2.17]{RSY$_1$}, \cite[Theorem~4.2]{RSY$_2$}).
Their use may turn out to be more expedient, it may be connected
with theorems on the existence of solutions of differential
equations in partial derivatives. At the same time, inner and outher
dilatations, which also play a role in a similar context, may not
satisfy the appropriate conditions necessary for their application.

The main goal of this manuscript is to obtain new estimates of the
distortion of the modulus of families of paths for using tangential
dilatations. The article is divided into two parts. The first part
concerns the inverse estimates of the modulus of families of paths,
the second part deals with direct estimates. The results of the
article mainly are proved for mappings with branching.

\medskip
Here are the necessary definitions. Let $X$ and $ Y $ be two spaces
with measures $\mu$ and $\mu^{\,\prime},$ respectively. We say that
a mapping $f:X\rightarrow Y$ has {\it $N$-property of Luzin}, if
from the condition $\mu(E)=0$ it follows that
$\mu^{\,\prime}(f(E))=0.$ Similarly, we say that a mapping
$f:X\rightarrow Y$ has {\it $N^{\, \prime}$-Luzin property,} if from
the condition $\mu^{\,\prime}(E)=0$ it follows that
$\mu(f^{\,-1}(E))=0.$ Given a mapping $f:D\,\rightarrow\,{\Bbb C},$
a set $E\subset D\subset {\Bbb C}$ and a point $y\,\in\,{\Bbb C},$
er define {\it a multiplicity function $N(y,f,E)$} as a number of
pre-images of a point $y$ in  $E,$ i.e.,
$$
N(y,f,E)\,=\,{\rm card}\,\left\{z\in E: f(z)=y\right\}\,,
$$
\begin{equation}\label{eq15}
N(f,E)\,=\,\sup\limits_{y\in{\Bbb C}}\,N(y,f,E).
\end{equation}
A mapping $f:D\rightarrow {\Bbb C}$ is called {\it discrete} if the
image $\{f^{-1}\left(y\right)\}$ of each point $y\,\in\ ,{\Bbb C}$
consists of isolated points, and {\it open} if the image of any open
set $U\subset D$ is an open set in ${\Bbb C}.$ A mapping $f$ of $D$
onto $D^{\,\prime}$ is called {\it closed} if $f(E)$ is closed in
$D^{\,\prime}$ for any closed set $E\subset D$ (see, e.g.,
\cite[chap.~3]{Vu}). Observe that, $N(f, D)<\infty$ for any open
discrete and closed mappings $f$ of a domain $D$ (see
\cite[Lemma~3.3]{MS}).

\medskip
Let $\Gamma$ be a family of paths $\gamma:(a, b)\rightarrow {\Bbb
C}$ (or dashed lines $\gamma:\bigcup\limits_{i=1}^{\infty} (a_i,
b_i)\rightarrow {\Bbb C}$). A Borel function $\rho:{\Bbb
C}\rightarrow\overline{{\Bbb R}^+}$ is called {\it admissible} for
the family $\Gamma,$ write $\rho\in{\rm adm}\,\Gamma,$ if
\begin{equation}\label{eq8.2.6}
\int\limits_{\gamma}\rho(z)\,|dz|\, \geqslant 1
\end{equation}
for any locally rectifiable path (dashed line) $\gamma\in\Gamma.$
Given $p\in(1,\infty),$ the quantity
$$M_p(\Gamma)=\inf\limits_{\rho\in{\rm adm}\,\Gamma} \int\limits_{\Bbb
C}\rho^p(z)\,dm(z)$$
is called {\it $p$-modulus} of $\Gamma.$ Let $z_0\in {\Bbb C},$
$0<r_1<r_2<\infty$ and
\begin{equation}\label{eq1ED}
B(z_0, r)=\{z\in {\Bbb C}: |z-z_0|<r\}\,,\qquad S(z_0,r) = \{
z\,\in\,{\Bbb C} : |z-z_0|=r\}\,,\end{equation}
\begin{equation}\label{eq1**A} A=A(z_0, r_1, r_2)=\{ z\,\in\,{\Bbb C} :
r_1<|z-z_0|<r_2\}\,.
\end{equation}
Given sets $E,$ $F\subset\overline{\Bbb C}$ and a domain $D\subset
{\Bbb C}$ we denote $\Gamma(E,F,D)$ a family of all paths
$\gamma:[a,b]\rightarrow \overline{\Bbb C}$ such that $\gamma(a)\in
E,\gamma(b)\in\,F$ and $\gamma(t)\in D$ for $t \in [a, b].$ If
$f:D\rightarrow {\Bbb C},$ $z_0\in
\overline{f(D)}\setminus\{\infty\}$ and
$0<r_1<r_2<r_0=\sup\limits_{y\in f(D)}|y-y_0|,$ then we define
$\Gamma_f(y_0, C_1, C_2)$ a family of all paths $\gamma$ in $D$ such
that $f(\gamma)\in \Gamma(C_1, C_2, A(y_0,r_1,r_2)).$ Let $Q_*:{\Bbb
C}\rightarrow [0, \infty]$ be a Lebesgue measurable function. We say
that {\it $f$ satisfies the inverse Poletsky inequality at $y_0\in
\overline{f(D)}\setminus\{\infty\}$ with respect to
$\alpha$-modulus,} $\alpha\geqslant 1,$ if the relation
\begin{equation}\label{eq2*A}
M_{\alpha}(\Gamma_f(y_0, C_1, C_2))\leqslant
\int\limits_{A(y_0,r_1,r_2)\cap f(D)} Q_*(y)\cdot \eta^{\alpha}
(|y-y_0|)\, dm(y)
\end{equation}
for any $0<r_1<r_2<d_0:={\rm dist}\, (z_0,
\partial D),$ all continua $C_1\subset
\overline{B(z_0, r_1)},$ $C_2\subset D\setminus B(z_0, r_2)$ and any
Lebesgue measurable function $\eta:(r_1, r_2)\rightarrow [0,
\infty]$ such that
\begin{equation}\label{eqA2}
\int\limits_{r_1}^{r_2}\eta(r)\, dr\geqslant 1\,.
\end{equation}
Let $y\in {\Bbb C}$ does not belong to the set $f(A),$ where $A$ is
the set of all points $z\in D$ where the mapping $f:D\rightarrow
{\Bbb C}$ has no a total differential, or  $J(z, f)=0.$ Given $p>1,$
we set
\begin{equation}\label{eq1}
Q(y):=K_{CT, p, y_0}(y, f):=\sum\limits_{z\in
f^{\,-1}(y)}\frac{\left(\biggl|(f^{\,\prime}(z))^{\,-1}\frac{y-y_0}{|y-y_0|}\biggr|\right)
^{p}}{|J(z, f)|}\,,
\end{equation}
The will call the function $K_{CT, p, y_0}$ in~(\ref{eq1}) is {\it
co-tangential dilatation of $f$ of the order $p$ at $y_0$}. The
following theorem holds.

\medskip
\begin{theorem}\label{th1} {\sl\, Let $1<\alpha\leqslant 2,$ let $y_0\in
\overline{f(D)}\setminus\{\infty\},$ $r_0=\sup\limits_{y\in
f(D)}|y-y_0|>0,$ and let $f:D\rightarrow{\Bbb C}$ be an open
discrete and closed mapping that is differentiable almost everywhere
and has $N$-Luzin property with respect to the Lebesgue measure in
${\Bbb C}.$ Assume that $\overline{D}$ is compact in ${\Bbb C },$ at
the same time, $m(f(B_*))=0,$ where $B_*$ is a set of points $z\in
D$ where $f$ has a total differential and $J(z, f)= 0.$ Suppose
that, any path $\alpha$ with $f\circ \alpha\subset S(y_0, r),$
$S(y_0, r)=y_0+re^{i\varphi/r},$ $\varphi\in [0, 2\pi r),$ is
locally rectifiable for almost all $r\in (\varepsilon,
\varepsilon_0)$ and, in addition, $f$ has $N^{\,- 1}$-property on
$S(y_0, r)\cap f(D)$ for almost all $r\in(\varepsilon, r_0)$ with
respect to ${\mathcal H}^{1}$ on $ S(y_0, r).$ If
$$Q(y):=K_{CT, \frac{\alpha}{\alpha-1}, y_0}(y, f):=\sum\limits_{z\in
f^{\,-1}(y)}\frac{\left(\biggl|(f^{\,\prime}(z))^{\,-1}\frac{y-y_0}{|y-y_0|}\biggr|\right)
^{\frac{\alpha}{\alpha-1}}}{|J(z, f)|}\in L^{\alpha-1}(f (D))\,,$$
then $f$ satisfies inverse Poletsky's inequality with respect to
$\alpha$-modulus at the point $y_0$ for
$$Q_*(y):=N^{\alpha}(f, D)\cdot Q^{\alpha-1}(y).$$}
\end{theorem}
Let us move on to the formulation of the results regarding the
analogue of Poletsky's inequality. Let $z\in D\subset {\Bbb C}$ be a
point where $f:D\rightarrow {\Bbb C}$ has partial derivatives with
respect to $x$ and $y,$ where $z=x+iy,$ $i^2=-1,$ and $J(z, f )\ne
0,$ where $J(z, f)$ denotes the Jacobian of the mapping $f$ at the
point $z.$ Let $z_0\in \overline{D}$ and let $p>1.$ Then we put
\begin{equation}\label{eq19}
K_{T, p, z_0}(z,
f):=\frac{\left(\biggl|(f^{\,\prime}(z))\frac{z-z_0}{|z-z_0|}\biggr|\right)
^{p}}{|J(z, f)|}
\end{equation}
at the non-degenerate differentiability point $z$ of $f,$ and $K_{T,
p, z_0}(z, f)=0$ otherwise. The quantity $K_{T, p, z_0}(z, f)$
in~(\ref{eq19}) is called {\it the tangential dilation of order $p$
of the mapping $f$ at the point $z$ relative to the point $z_0.$}

\medskip
\begin{remark}\label{rem1}
We set
\begin{equation}\label{eq20}
D_{f}(z,
z_0)=\frac{\left|1-\frac{\overline{z-z_0}}{z-z_0}\mu(z)\right|^2}{1-|\mu(z)|^2}\,,
\end{equation}
at the points $z$ of the non-degenerate differentiability of $f,$
where $\mu(z)=\frac{f_{\overline z}}{f_z},$ $f_z\ne 0,$ $\mu(z)=0$
при $f_z=0,$ $f_z=(f_{x}-if_{y})/2,$
$f_{\overline{z}}=(f_{x}+if_{y})/2,$ $z=x+iy,$ $i^2=-1$ (see
e.g.~\cite[Lemma~2.10]{RSY$_1$} or \cite[Lemma~11.2]{MRSY}). Let us
note the following remark concerning the connection of quantities
in~(\ref{eq1}) and~(\ref{eq20}).

\medskip
It may be shown that the relation
\begin{equation}\label{eq17}
K_{CT, 2, z_0}(f(z), f^{\,-1})=D_f(z, z_0)\,,\quad y_0=f(z_0)\,,
\end{equation}
holds whenever $f$ is a homeomorphism and $f$ is non-degenerate
differentiable at $z\in D,$ see~\cite[formulae~(11.35),
(11.43)]{MRSY}. Moreover, for $p=2,$ $D_f(z, z_0)=K_{T, 2, z_0}(z,
f)$ at the non-degenerate differentiability points
(see~\cite[relation~(11.35)]{MRSY}).
\end{remark}

\medskip In what follows, $C^k_0(U)$ denotes the space of functions
$u:U \rightarrow {\Bbb R} $ with by a compact carrier in $U,$ which
have $k$ partial continua derivatives in $U.$ Recall the concept of
Sobolev classes, see \cite[Section~2, Ch.~I]{Re}. Let $U$ be an open
set, $U\subset{\Bbb C},$ $u:U\rightarrow {\Bbb R}$ is some function
such that $u \in L_{\rm loc}^{\,1}(U). $ Suppose there is a function
$v\in L_{\rm loc}^{\,1}(U)$ such that equality
$$\int\limits_U \frac{\partial \varphi}{\partial x_i}(z)u(z)\,dm(z)=
-\int\limits_U \varphi(z)v(z)\,dm (z)$$
is performed for any function $\varphi\in C^1_{\,0}(U),$ $i=1,2.$ In
this case, we will say that the function $v$ is a generalized
partial derivative of the first order function $u$ with respect to
$x_i$ and denote it by $\frac{\partial u}{\partial x_i}(z):= v.$ A
function $u\in W_{\rm loc}^{1,1}(U),$ if $u$ has generalized partial
derivatives with respect to all variables in $U,$ are locally
integrable in $U.$

\medskip
A mapping $f:D\rightarrow {\Bbb C},$ $f(z)=u(z)+iv(z),$ belongs to
to {\it Sobolev class} $W_{\rm loc}^{1,1},$ write $f \in
W^{1,1}_{\rm loc}(D),$ if $u$ and $v$ have generalized partial
derivatives of the first order, locally integrable in $D.$ We write
$f\in W^{1, k}_{\rm loc}(D),$ $k\in {\Bbb N},$ if $u$ and $v,$ are
also locally integrable in degree $k.$

\medskip
Recall that a mapping $f$ between domains $D$ and $D^{\,\prime}$ has
{\it a finite distortion,} if $f\in W^{1,1}_{\rm loc}$ and, in
addition, there exists a function $K(z),$ $K(z)<\infty$ a.e., such
that
$$
{\Vert f^{\,\prime}(z)\Vert}^2\leqslant K(z)\cdot J(z, f)
$$
for almost all $z\in D,$ where $\Vert
f^{\,\prime}(z)\Vert=|f_z|+|f_{\overline{z}}|.$

\medskip
Let $Q:{\Bbb C}\rightarrow {\Bbb R}$ be a Lebesgue measurable
function such that $Q(z)\equiv 0$ for $z\in{\Bbb C}\setminus D.$ Let
$z_0\in\overline{D},$ $z_0\ne\infty.$
Given $\alpha\geqslant 1,$ a mapping $f:D\rightarrow \overline{\Bbb
C}$ is called {\it a ring $Q$-mapping at a point $z_0\in
\overline{D}\setminus \{\infty\}$ with respect to $\alpha$-modulus,}
if the condition
\begin{equation} \label{eq2*!A}
M_{\alpha}(f(\Gamma(C_1, C_2, D)))\leqslant \int\limits_{A\cap D}
Q(z)\cdot \eta^{\,\alpha} (|z-z_0|)\, dm(z)
\end{equation}
holds for any $0<r_1<r_2<d_0:={\rm dist}\, (z_0,
\partial D),$ all continua $C_1\subset
\overline{B(z_0, r_1)},$ $C_2\subset D\setminus B(z_0, r_2)$ and any
Lebesgue measurable function $\eta:(r_1, r_2)\rightarrow [0,
\infty]$ such that
\begin{equation}\label{eq8BC}
\int\limits_{r_1}^{r_2}\eta(r)\,dr\geqslant 1\,.
\end{equation}

\medskip
Recall that, a pair $E=(A,\,C)$ is said to be a {\it condenser,} if
$A$ is an open set in ${\Bbb C},$ and $C$ is a non-empty compact
subset of $A.$ Denote by $dm(z)$ the element of the Lebesgue measure
in ${\Bbb C},$ $W_0(E)=W_0\left(A,\,C\right)$ is a family of all
absolutely continuous on lines (ACL) functions $u:A\ rightarrow
{\Bbb R}$ with a compact support in $A$ such that $u(z)\geqslant 1$
on $C.$ The quantity
\begin{equation}\label{eq4G}{\rm cap}_p\,E\quad=\quad{\rm
cap}_p\,(A,\,C)=\inf\limits_{u\,\in\,W_0\left(E\right)
}\quad\int\limits_A\,|\nabla u|^p\,\,dm(z)\,,
\end{equation}
is called {\it $p$-capacity} of $E$. The following result holds.

\medskip
\begin{theorem}\label{th1B}
{\sl\, Let $f:D\rightarrow {\Bbb C}$ be a homeomorphism with a
finite distortion and let $1<\alpha\leqslant 2.$ Set
$p=\frac{\alpha}{\alpha-1}.$ Suppose that $K^{\alpha-1}_{T, p,
z_0}(z, f)\in L^1(D),$ where $K_{T, p, z_0}(z, f)$ is defined
in~(\ref{eq19}). Then $f$ satisfies the relation~(\ref{eq2*!A}) at
the point $z_0\in \overline{D}\setminus \{\infty\}$ at
$Q(z)=K^{\frac{1}{p-1}}_{T, p, z_0}(z, f).$}
\end{theorem}

\medskip
\begin{theorem}\label{th1A}{\sl\,
Let $f:D\rightarrow {\Bbb C}$ be an open, discrete and closed
bounded mapping with finite distortion such that $N(f, D)<\infty.$
Let $1< \alpha\leqslant 2$ and $z_0\in D.$ Set
$p=\frac{\alpha}{\alpha-1}.$ Suppose that $K_{T, p, z_0}(z, f)\in
L_{\rm loc}^{\alpha-1}(D),$ where $K_{T, p, z_0}(z, f)$ is defined
in~(\ref{eq19}). Then $f$ satisfies the relation
$${\rm cap}_{\alpha}\, f(\mathcal{E})\leqslant
\int\limits_{A} N^{\alpha-1}(f, D)K^{\alpha-1}_{T, p, z_0}(z,
f)\cdot \eta^{\alpha}(|z-z_0|)\, dm(z)
$$
for $\mathcal{E}=(B(z_0, r_2), \overline{B(z_0, r_1)}),$ $A=A (z_0,
r_1, r_2),$ $0<r_1<r_2<\varepsilon_0:={\rm dist}\,(z_0,
\partial D),$ and any Lebesgue measurable function $\eta \colon (r_1,r_2)\rightarrow [0,\infty ]$
which satisfies the relation~(\ref{eq8BC}).}
\end{theorem}

\medskip
\begin{theorem}\label{th2}{\sl\,
Let $f:D\rightarrow {\Bbb C}$ be an open, discrete and closed
mapping with finite distortion, and $1<\alpha\leqslant 2.$ Suppose
that $z_0\in \partial D,$ $p=\frac{\alpha}{\alpha-1}$ and that
$K_{T, p, z_0}(z, f)\in L^{\alpha-1}(D),$ where $K_{T, p, z_0}(z,
f)$ is defined in~(\ref{eq19}). Then, for any
$\varepsilon_0<d_0:=\sup\limits_{z\in D}|z-z_0|$ and a compact set
$C_2\subset D\setminus B(z_0, \varepsilon_0)$ there is
$\varepsilon_1,$ $0<\varepsilon_1<\varepsilon_0,$ such that the
relation
\begin{equation}\label{eq3A}M_{\alpha}(f(\Gamma(C_1, C_2, D)))\leqslant \int\limits_{A(z_0,
\varepsilon, \varepsilon_1)} N^{\alpha-1}(f, D)K^{\alpha-1}_{T, p,
z_0}(z, f) \eta^{\alpha}(|z-z_0|)\,dm(z)
\end{equation}
holds for any $\varepsilon\in (0, \varepsilon_1)$ and any
$C_1\subset \overline{B(z_0, \varepsilon)}\cap D,$ where $A(z_0,
\varepsilon, \varepsilon_1)$ is defined in~(\ref{eq1**A}), and
$\eta: (\varepsilon, \varepsilon_1)\rightarrow [0,\infty]$ is
arbitrary Lebesgue measurable function that satisfies the
relation~(\ref{eq8BC}). }
\end{theorem}

\section{Preliminaries}

The following important information concerning the capacity of a
pair of sets with respect to a domain can be found in the paper of
W.~Ziemer \cite{Zi$_1$}. Let $G$ be a bounded domain in ${\Bbb C}$
and $C_0 , C_1$ be non-intersecting compact sets, which belong to
the closure $G.$ Let us set $R=G \setminus (C_{0} \cup C_{1})$ and
$R^{\,*}=R \cup C_{0}\cup C_{1} .$ Given $p>1,$ we define the {\it
$p$-capacity of the pair $C_0, C_1$ relative to the closure $G$} by
the equality
$$C_p[G, C_0, C_1] = \inf \int\limits_{R} |\nabla u|^p\, dm(z),$$
where the exact lower bound is taken over all functions $u,$
continuous in $R^{\,*},$ $u\in ACL(R),$ such that $u=1$ on $C_1$ and
$u=0 $ on $C_0.$ The specified functions are called {\it admissible}
for $C_p[G, C_0,C_1].$ We will say that the set {\it $\sigma \subset
{\Bbb C}$ separates $C_0$ from $C_1$ in $R^{\,*},$} if $\sigma \cap
R$ is closed in $R$ and there are disjoint sets $A$ and $B,$ open
relative to $R^{\,*} \setminus \sigma,$ such that $R^{\,*} \setminus
\sigma =A\cup B,$ $C_{0}\subset A$ and $C_{1} \subset B.$ Let
$\Sigma$ denote the class of all sets that separate $C_0$ and $C_1$
in $R^{\,*}.$ Given a number $p^{\prime} = p/(p- 1),$ we define the
quantity
\begin{equation}\label{eq13.4.12}
\widetilde{M_{p^{\prime}}}(\Sigma)=\inf\limits_{\rho\in
\widetilde{\rm adm} \Sigma} \int\limits_{\Bbb
C}\rho^{\,p^{\prime}}dm(z)
\end{equation}
where the inclusion $\rho\in \widetilde{\rm adm}\,\Sigma$ denotes
that $\rho$ is an Borel measurable function in ${\Bbb C}$ such that
\begin{equation} \label{eq13.4.13}
\int\limits_{\sigma \cap R}\rho\, d{\mathcal H}^{1} \geqslant
1\quad\forall\, \sigma \in \Sigma\,. \end{equation}
Note that, by Ziemer's result
\begin{equation}\label{eq3}
\widetilde{M_{p^{\,\prime}}}(\Sigma)=C_p[G , C_0 ,
C_1]^{\,-1/(p-1)}\,,
\end{equation}
see~\cite[theorem~3.13]{Zi$_1$} for $p=2$ and \cite[p.~50]{Zi$_2$}
for $1<p<\infty.$ In addition, by the Hesse's result
\begin{equation}\label{eq4}
M_p(\Gamma(E, F, D))= C_p[D, E, F]\,,
\end{equation}
where $(E \cup F)\cap
\partial D = \varnothing$ (see~\cite[Theorem~5.5]{Hes}). Shlyk
showed that the requirement $(E \cup F)\cap
\partial D = \varnothing$ can be removed,
in other words, equality~(\ref{eq4}) holds for arbitrary disjoint
nonempty compact sets $E, F\subset \overline{D}$
(see~\cite[Theorem~1]{Shl}).

\medskip
We say that some property $P$ is satisfied for {\it $p$-almost all
paths} in the domain $D,$ if this property is satisfied for all
paths in $D,$ except, perhaps, some of their subfamily, $p$-modulus
of which equals to zero. We say that a Lebesgue measurable function
$\rho\colon{\Bbb C}\rightarrow\overline{{\Bbb R}^+}$ is {\it
$p$-extensively admissible} for a family $\Gamma$ of paths $\gamma$
(or dashed lines) in ${\Bbb C},$ abbr. $\rho\in{\rm ext}_p\,{\rm
adm}\,\Gamma,$ if the relation~(\ref{eq8.2.6}) holds for $p$-a.e.
paths (dashed lines) $\gamma$ of $\Gamma.$

\medskip
Let $E$ be a set in ${\Bbb C}$ and let $\gamma :\Delta\rightarrow
{\Bbb C}$ be some path. Denote by $\gamma\cap
E\,=\,\gamma(\Delta)\cap E.$ Let the path $\gamma$ be locally
rectifiable and the length function $l_{\gamma}(t)$ is defined
above. Let us put
$$
l(\gamma\cap E):= {\rm mes}_1\,(E_ {\gamma}), \quad E_ {\gamma} =
l_{\gamma}(\gamma ^{\,-1}(E))\,.
$$
Here, as everywhere above, ${\rm mes}_1\,(A)$ denotes the length
(linear Lebesgue measure) of the set $A\subset {\Bbb R}.$ Note that
$$E_ {\gamma}=\gamma_0^{\,-1}\left(E\right)\,,$$
where $\gamma_0:\Delta _{\gamma}\rightarrow {\Bbb C}$ is natural
parametrization of the path $\gamma,$ and that
$$l(\gamma\cap E) = \int\limits_{\gamma} \chi_E(z)\,|dz|
= \int\limits_{\Delta _{\gamma}} \chi _{E_\gamma }(s)\,dm_1(s)\,.$$

\medskip
The following statement can be found in~\cite[Lemma~8.1]{MRSY}.

\medskip
\begin{proposition}\label{pr2}{\sl\,
Let $E$ be a subset of a domain $D\subset{\Bbb C},$ $n\geqslant 2,$
$p\geqslant 1.$ Then the set $E$ is Lebesgue measurable if and only
if the set $\gamma\cap E$ is measurable for $p$-almost all paths
$\gamma$ in $D.$ Moreover, $m(E)=0$ if and only if
\begin{equation}\label{eq1B}
l(\gamma\cap E)=0
\end{equation}
for $p$-almost all curves $\gamma$ in $D.$}
\end{proposition}

\medskip
\begin{remark}\label{rem2}
Recall that the family of paths $\Gamma_1$ is called {\it shorter}
compared to the family of paths $\Gamma_2,$ is written
$\Gamma_1<\Gamma_2,$ if for each path $\gamma_2\in \Gamma_2,$
$\gamma_2:(a_2, b_2)\rightarrow {\Bbb C},$ there is a path
$\gamma_1\in \Gamma_1,$ $\gamma_1:(a_1, b_1)\rightarrow {\Bbb C},$
such that $\gamma_1(t)=\gamma_2(t)$ for all $t\in (a_1, b_1)\subset
(a_2, b_2).$ From the definition of the modulus of families of paths
it follows that the condition $\Gamma_1<\Gamma_2$ implies that
$M_p(\Gamma_1)\geqslant M_p(\Gamma_2)$ (see, e.g.,
\cite[Theorem~6.2]{Va}).

From these considerations it also follows that the
Proposition~\ref{pr2} is true also for families of dashed lines, not
only families of paths. Really, let the set $E$ be Lebesgue
measurable, and let $\Gamma_0$ be a family of all dashed lines
$\gamma:\bigcup\limits_{i=1}^{\infty}(a_i, b_i)\rightarrow {\Bbb
C},$ for which the measurability of the set $\gamma\cap E$ is
violated. Then, by the countable additivity of the Lebesgue measure
there is a family $\Gamma_1$ of paths $\gamma_k:(a_k,
b_k)\rightarrow{\Bbb C}$ such that $\gamma_k\cap E$ is not
measurable. However, from the definition of the modulus of families
of paths (see the considerations given above), we have:
$M_p(\Gamma_0)\leqslant M_p(\Gamma_1)=0,$ i.e., $M_p(\Gamma_0).$ The
inverse statement, i.e., that , that the dimensionality of the set
$E$ follows from the dimensionality of $\gamma\cap E$ for almost all
of dashed lines $\gamma,$ is obvious and follows directly from the
Proposition~\ref{pr2}.

The validity of the equality~(\ref{eq1B}) for families of dashed
lines is proved in the same way as the equivalence of the
dimensionality of the set $\gamma\cap E$ to the dimensionality of
the set $E$ itself.
\end{remark}

\medskip
The following statement is proved in~\cite[lemma~4.1]{IS}.

\medskip
\begin{proposition}\label{pr7}
{\sl\, Let $D$  be a domain in ${\Bbb C},$ $p>1,$ and $z_0\in D.$ If
some property $P$ holds for $p$-almost circles $D(z_0, r):=S(z_0,
r)\cap D,$ $r\in (\varepsilon, \varepsilon_0),$ where ''$p$-almost
all'' should be understood in the sense of the modulus of families
of dashed lines lines, in addition, the set
$$E=\{r\in {\Bbb R}\,:\,
P\qquad \text{holds\,\, for}\qquad S(z_0, r)\cap D\}$$
is Lebesgue measurable. Then $P$ also holds for almost all $D(z_0,
r)$ with respect to the parameter $r\in (\varepsilon,
\varepsilon_0).$ On the contrary, if $P$ holds for almost all
circles $D(z_0, r):=S(z_0, r)\cap D$ with respect to the Lebesgue
measure by $r\in (\varepsilon, \varepsilon_0),$ then $P$ also holds
for $p$-almost all of dashed lines $D(z_0, r):=S(z_0, r)\cap D$ in
the sense of $p$-modulus for any $p>1.$}
 \end{proposition}

\medskip
The following statement can be found in~\cite[lemma~3.7]{Vu}.

\medskip
\begin{proposition}\label{pr1}
{\sl\, Let $f:G\rightarrow {\Bbb C}$ be an open, discrete and closed
mapping, let $\beta:[a, b)\rightarrow f(G)$ be a path and let
$k=N(f, G).$ Then there are paths $\alpha_j:[a, b)\rightarrow G,$
$1\leqslant j\leqslant k,$ such that: (1) $f\circ \alpha_j=\beta,$
(2) ${\rm card}\,\{j: \alpha_j(t)=z\}=i(z, f)$ at $z\in
f^{\,-1}(|\beta|)$ and $t\in [a, b),$ where $i(z, f)$ is the local
topological index of the mapping $f$ at points $z$ and (3)
$\sum\limits_{j=1}^k |\alpha_j|=f^{-1}(|\beta|).$}
\end{proposition}

\medskip
We write $\alpha\subset \beta,$ if $\beta|_{J}=\alpha,$ where
$\beta:I\rightarrow {\Bbb C}$ and $J$ is some interval, segment, or
the semiinterval $I.$ Let $f:D\rightarrow {\Bbb C}$ be a discrete
mapping, $\beta:I_0\rightarrow {\Bbb C}$ be a closed rectifiable
path, and $\alpha:I \rightarrow D$ is a path such that $f\circ
\alpha\subset \beta.$ If the length function
$l_{\beta}:I_0\rightarrow [0, l(\beta)]$ is constant on a certain
interval $J\subset I,$ then $\beta$ is constant on $J$ and due to
the discreteness of the mapping $f,$ the path $\alpha$ is also
constant on $J.$ Therefore, there exists a unique function
$\alpha^{\,*}:l_\beta(I)\rightarrow D$ such that $\alpha=\alpha^{\,
*}\circ (l_\beta|_I).$ We say that $\alpha^{\,*}$ is the {\it
$f$-representation of the path $\alpha$ with respect to $\beta.$
Similarly, one can define $f$-representations for arbitrary
rectifiable paths, not only closed ones. }

\medskip
The proof of the following lemma is based on the approach used when
establishing lower estimates of the distortion of the modulus of
families paths, (see, e.g., \cite[Theorem~2.17]{RSY$_1$} and
\cite[Theorem~4.1]{RSY$_2$}; see also~\cite[Theorem~5]{KRSS} and
\cite[Theorem~4]{Sev$_1$}).

\medskip
\begin{lemma}\label{lem1} {\sl\, Let $p>1,$ $D$ be a domain in ${\Bbb C},$
$f:D\rightarrow {\Bbb C}$ be an open, discrete and closed mapping
which is differential almost everywhere and has $N$-property of
Luzin with respect to the Lebesgue measure in ${\Bbb C}.$ Let
$y_0\in \overline{f(D)}\setminus\{\infty\},$ $r_0=\sup\limits_{y\in
f( D)}|y-y_0|>0,$ $0<\varepsilon<\varepsilon_0<r_0.$ Assume that,
$m(f(B_*))=0,$ where $B_*$ is the set of all points $z$ in which $f$
has a total differential and $J(z, f)=0.$ Denote by
$\Sigma_{\varepsilon}$ the family of all sets of the form
\begin{equation}\label{eq12}
\{f^{\,-1}(S(y_0, r))\},\quad r\in (\varepsilon, \varepsilon_0)\,.
\end{equation}
Assume that, for almost all $r\in (\varepsilon, \varepsilon_0),$
$f$-has the $N^{\,-1}$-property on $S(y_0, r)$  with respect to the
Hausdorff  measure ${\mathcal H}^{1}$ on $S(y_0, r)$ and in
addition, any path $\alpha$ with $f\circ \alpha\subset S(y_0, r)$ is
locally rectifiable. Then
\begin{equation}\label{eq13}
\widetilde{M_{p}}(\Sigma_{\varepsilon})\geqslant\frac{1}{N^{p}(f,
D)}\inf\limits_{\rho\in{\rm
ext\,adm}_p\,f(\Sigma_{\varepsilon})}\int\limits_{f(D)\cap A(y_0,
\varepsilon, \varepsilon_0)}\frac{\rho^p(y)}{Q(y)}\,dm(y)\,,
\end{equation}
where $Q$ is defined by the relation
$$Q(y):=K_{CT, p, y_0}(y, f):=\sum\limits_{z\in
f^{\,-1}(y)}\frac{\left(\biggl|(f^{\,\prime}(z))^{\,-1}\frac{y-y_0}{|
y-y_0|}\biggr|\right) ^{p}}{|J(z, f)|}\,.$$ }
\end{lemma}

\medskip
\begin{proof} We denote by $B$ the (Borel)
set of all points $z\in D$, where the mapping $f$ has the total
differential $f^{\,\prime}(z)$ and $J(z, f)\ne 0 .$ By Kirsbraun's
theorem and by the uniqueness of the approximative differential
(see, e.g., \cite[2.10.43 and Theorem~3.1.2]{Fe}) it follows that
the set $B$ is a countable union Borel sets $B_k$, $k=1,2,\ldots,$
such that the mapping $f_k=f|_{B_k}$ are bilipschitz homeomorphisms
(see \cite[Lemma~3.2.2 and Theorems~3.1.4 and 3.1.8]{Fe}). Without
loss of generality, we may assume that the sets $B_k$ are pairwise
disjoint. We also denote by $B_*$ the set of all points $z\in D,$
where $f$ has a total differential, however $J(z, f)=0.$

\medskip
Since the set $B_0:=D\setminus (B\cup B_*)$ has a Lebesgue measure
zero, and $f$ has the $N$-Luzin property, then $m(f(B_0))=0.$ Now,
by Proposition~\ref{pr2} and Remark~\ref{rem2} $l(S_r\cap f(B_0))=0$
for $p$-almost all circles $S_r:=S(y_0,r)\cap f(D)$ centered at the
point $y_0,$ where ''almost all'' should be understood in the sense
of $p$-modulus of families of paths. Similarly, $l(S_r\cap
f(B_*))=0$ for $p$-almost all cuch circles of $S_r.$

\medskip
Let us fix the circle
$$\widetilde{\gamma}(t)=S(y_0,r)=re^{it}+y_0, \quad t\in [0, 2\pi)\,.$$
Then $S_r:=S(y_0,r)\cap f(D)$ is a family of dashed lines
$$\widetilde{\gamma}_i: \bigcup\limits_{i=1}^{\infty} (a_i,
b_i)\rightarrow f(D),\quad (a_i, b_i)\subset {\Bbb R},\quad
i=1,2,\ldots .\,$$
Let us denote this family of dashed lines by $\Gamma.$ Now we fix
some point $\omega^i_0\in f(D)$ which belongs to a dashed line
$\widetilde{\gamma}_i: (a_i, b_i)\rightarrow f(D)$ such that
$\widetilde{\gamma}_i(t_i)=\omega^i_0$ for some $t_i\in (a_i, b_i).$
Note that $N(f, D)<\infty,$ where $N(f, D)$ is defined by the
ratio~(\ref{eq15}) (see~\cite[Lemma~3.3]{MS}).

\medskip
By Proposition~\ref{pr1} the path $\widetilde{\gamma}_i|_{[t_i,
b_i)}$ has $k:=N(f, D)$ total liftings $\gamma^1_{il}\rightarrow D,$
$1\leqslant l\leqslant k,$ such that ${\rm card}\,\{j:
\gamma^1_{il}(t)=z\}=i(z, f)$ at $z\in
f^{\,-1}(|\widetilde{\gamma}_i|_{[t_i, b_i)}|)$ and $t\in [t_i,
b_i),$ where $\sum\limits_{l=1}^k
|\gamma^1_{il}|=f^{\,-1}(|\widetilde{\gamma}_i|_{[t_i, b_i)}|),$
which start at some points $z_i,$ $1\leqslant i\leqslant k.$
Similarly, by Proposition~\ref{pr1} any path
$\widetilde{\gamma}_i|_{(a_i, t_i]}$ has $k:=N(f, D)$ complete
liftings $\gamma^2_{il}\rightarrow D,$ $1\leqslant l\leqslant k,$
such that ${\rm card}\,\{j: \gamma^2_{il}(t)=z\}=i(z, f)$ at $z\in
f^{\,-1}(|\widetilde{\gamma}_i|_{(a_i, t_i]}|)$ and $t\in (a_i,
t_i],$ where $\sum\limits_{l=1}^k
|\gamma^1_{il}|=f^{\,-1}(|\widetilde{\gamma}_i|_{(a_i, t_i]}|),$
starting at some points $z_i,$ $1\leqslant i\leqslant k.$ By uniting
the paths $\gamma^1_{il}$ and $\gamma^2_{il},$ we obtain the paths
$\gamma_{il}:(a_i, b_i)\rightarrow D,$ $l=1,2,\ldots, k,$ which have
the following properties:
$$(1)\qquad f\circ \gamma_{il}=\widetilde{\gamma}_i,\qquad l=1,2,\ldots, k\,,$$
$$(2)\qquad {\rm card}\,\{j: \gamma_{il}(t)=z\}=i(z, f)\qquad z\in
f^{\,-1}(|\widetilde{\gamma}_i|)\,,\qquad t\in (a_i, t_i]\,,$$
$$(3)\qquad \bigcup\limits_{l=1}^k
|\gamma_{il}|=f^{\,-1}(|\widetilde{\gamma}_i|)\,.$$\label{eq11A}
By the condition of the lemma, the path $\gamma^{*}_{il}(s)$ is
locally rectifiable for almost all $r.$ Then also
$\gamma^{*}_{il}(s)$ is locally rectifiable for $p$-almost all
$S_r\in \Gamma$ (see Proposition~\ref{pr7}). For the dashed lines
considered in the proof of Lemma~\ref{lem1} and their families, see
the following Figure and diagram.
\begin{figure}[h]
\centering\includegraphics[scale=0.6]{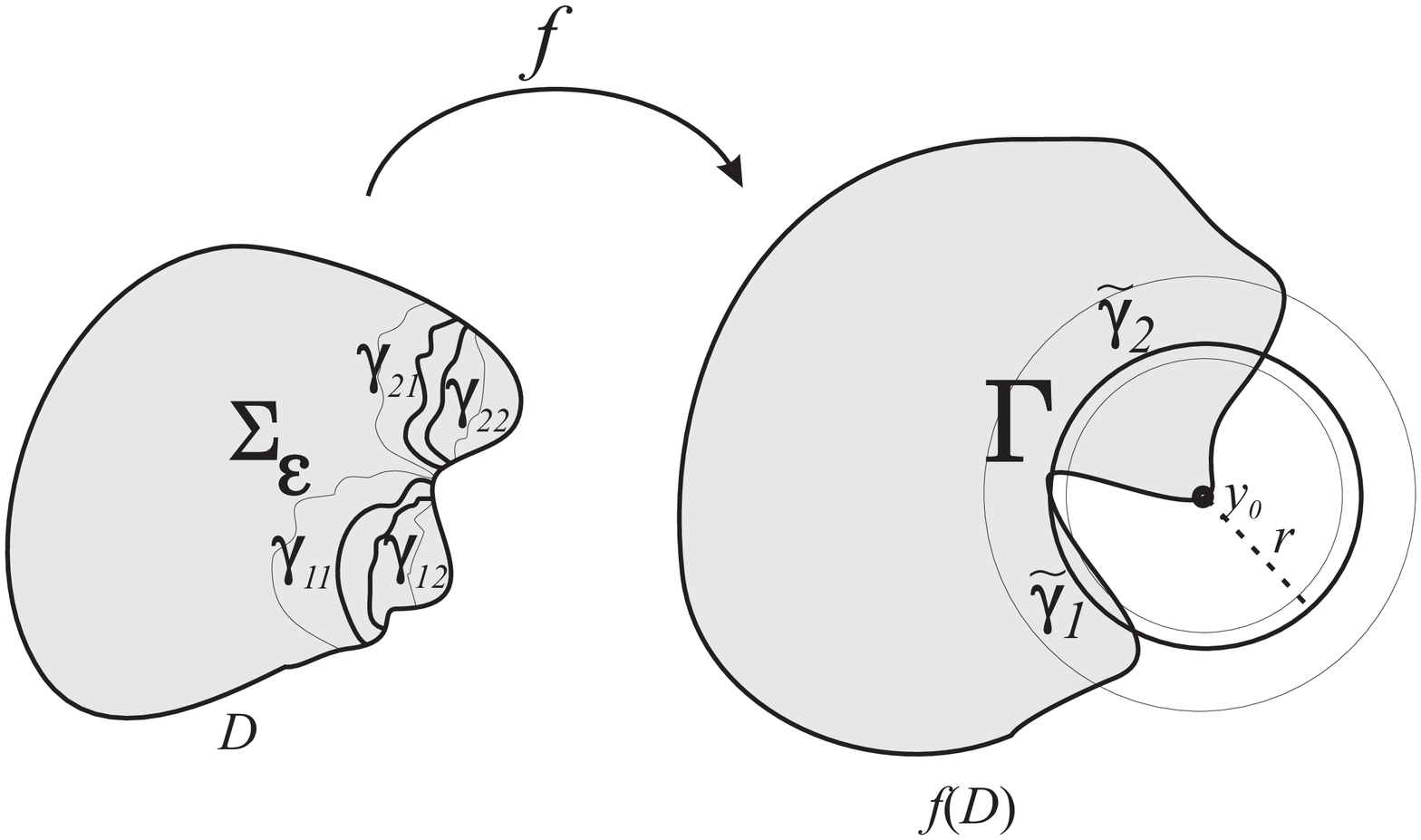}
\caption{To the proof of Lemma~\ref{lem1}} \label{fig2}
\end{figure}
\begin{figure}[h]
\centering\includegraphics[scale=0.7]{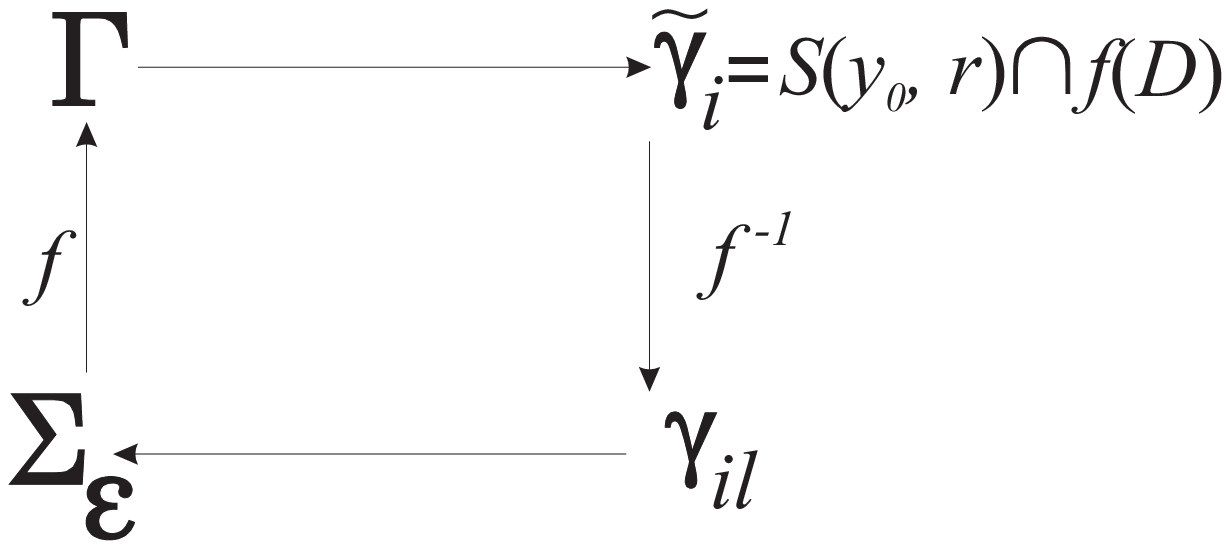}
\caption{To the proof of Lemma~\ref{lem1}} \label{fig3}
\end{figure}
Denote by $\Gamma_0$ the family of all $S_r,$ for which $l(S_r\cap
f(B_0))>0,$ or $l(S_r\cap f(B_*))>0,$ or the corresponding path
$\gamma^{*}_{il}(s)$ is not locally rectifiable. Now let $S_r\in
\Gamma\setminus \Gamma_0$ and let $\widetilde{\gamma}_i$ be
arbitrary dashed lines in $S_r.$ Let us first consider the case when
this a path is rectifiable for the same $r.$ Let $\rho\in
\widetilde{{\rm adm}}\,\Sigma_{\varepsilon}.$ We put
\begin{equation}\label{eq3C} \widetilde{\rho}(y)=\left \{\begin{array}{rr}
\sup\limits_{z\in
f^{\,-1}(y)}\rho(z)\biggl|(f^{\,\prime}(z))^{\,-1}\frac{y-y_0}{|y
-y_0|}\biggr|\,,
& y\in f(D\setminus B_0)\setminus f(B_*), \\
0, & \text{in\,\,other\,\, cases.}\end{array} \right.\,,
\end{equation}
Note that $\widetilde{\rho}(y)=\sup\limits_{k\in {\Bbb
N}}\widetilde{\rho}_k(y),$ where
\begin{equation}\label{eq3E} \widetilde{\rho}_k(y)=\left \{\begin{array}{rr}
\rho(f_k^{\,-1}(y))\biggl|(f^{\,\prime}(f_k^{\,-1}(y)))^{\,-1}\frac
{y-y_0}{|y-y_0|}\biggr|\,,
& y\in f(D\setminus B_0)\setminus f(B_*), \\
0, & \text{in\,\,other\,\, cases.}\end{array} \right.\,,
\end{equation}
therefore, the function $\widetilde{\rho}(y)$ is Borel, (see, e.g.,
\cite[Theorem~I~(8.5)]{Sa} or \cite[Section~2.3.2]{Fe}).
Denote, as usual, $\widetilde{\gamma}^0_i=\widetilde{\gamma}^0_i(s)$
is normal representation of the path $\widetilde{\gamma}_i,$ namely,
$$\widetilde{\gamma}^0(t)=\widetilde{\gamma}^0_i\circ s_i(t)\,,$$
were $s_i(t)$ denotes the length of $\widetilde{\gamma}$ on $(a_i,
b_i),$ and $s\in (0, l(\widetilde{\gamma}_i))$ and
$l(\widetilde{\gamma}_i)$ denotes the length of
$\widetilde{\gamma}_i.$ More precisely,
\begin{equation}\label{eq6}
\widetilde{\gamma}^0_i(s)=y_0+re^{i(s+s_i)/r}\,, \quad r=const,\quad
s\in (0, l(\widetilde{\gamma}_i))\,,
\end{equation}
where $s_i\in {\Bbb R}$ is some parameter such that
$\widetilde{\gamma}^0_i(s_i)=a_i.$

\medskip
Since $S_r\in \Gamma\setminus \Gamma_0,$ then
$\widetilde{\gamma}^0_i(s)\not\in f(B_0)$ for almost all $s\in (0,
l(\widetilde{\gamma}_i)).$ Then
\begin{equation}\label{eq2A}
\int\limits_{\widetilde{\gamma}_i}{\widetilde{\rho}}(y)\,|dy|=
\int\limits_0^{l(\widetilde{\gamma}_i)}\widetilde{\rho}(\widetilde{\gamma}^0_i(s))
\,ds=
\frac{1}{k}\sum\limits_{l=1}^k\int\limits_0^{l(\widetilde{\gamma}_i)}
\widetilde{\rho}(\widetilde{\gamma}^0_i(s))\,ds\,,
\end{equation}
where $k=N(f, D).$ Let the paths $\gamma_{il}$ satisfying properties
(1), (2) and (3) mentioned above, and let $\gamma^{* }_{il}(s)$ be
their $f$-images, that is,
$f(\gamma^{*}_{il}(s))=\widetilde{\gamma}^0_i(s)$ at all $s\in l(0,
l(\widetilde{\gamma}_i)).$ Since $\widetilde{\gamma}^0_i(s)\not\in
f(B_0)$ for almost all $s\in (0, l(\widetilde{\gamma}_i)),$ it
follows that $\gamma ^{*}_{il}(s)\not\in B_0$ for almost all $s\in
(0, l(\widetilde{\gamma}_i)).$ Therefore, points
$f^{\,-1}(\widetilde{\gamma}^0_i(s))=\{\gamma^{*}_{i1}(s),\gamma^{*}_{i2}
(s)\,,\ldots \, \gamma^{*}_{ik}(s)\}$ are different for almost all
$s\in (0, l(\widetilde{\gamma}_i)).$ Due to the definition
of~$\widetilde{\rho}$ in~(\ref{eq3C}), for each $1\leqslant
l\leqslant k$ we obtain that:
$$\int\limits_0^{l(\widetilde{\gamma}_i)}
\widetilde{\rho}(\widetilde{\gamma}^0_i(s))\,ds=
\int\limits_0^{l(\widetilde{\gamma}_i)} \sup\limits_{z\in
f^{\,-1}(\widetilde{\gamma}^0_i(s))}
\rho(z)\biggl|(f^{\,\prime}(z))^{\,-1}\frac{\widetilde{\gamma}^0_i(s)-y_0}
{|\widetilde{\gamma}^0_i(s)-y_0|}\biggr|\,ds\,\geqslant$$
$$\geqslant
\int\limits_0^{l(\widetilde{\gamma}_i)}\rho(\gamma^{*}_{il}(s))
\biggl|(f^{\,\prime}(\gamma^{*}_{il}(s)))^{\,-1}\frac{
\widetilde{\gamma}^0_i(s)-y_0}{|\widetilde{\gamma}^0_i(s)-y_0|}\biggr|\,ds=$$
\begin{equation}\label{eq3D}
= \int\limits_0^{l(\widetilde{\gamma}_i)}\rho(\gamma^{*}_{il}(s))
\biggl|(f^{\,\prime}(\gamma^{*}_{il}(s)))^{\,-1}(
e^{i(s+s_i)/r})\biggr|\,ds\,.
\end{equation}
Since $f(\gamma^{*}_{il}(s))=y_0+re^{i(s+s_i)/r},$ in addition,
$\gamma^{*}_{il}( s)\in D\setminus (B_0\cup B_*)$ for $p$-almost all
$\widetilde{\gamma}_i\in \Gamma$ and for almost all $s\in l(0,
l(\widetilde{\gamma}_i)),$ by the differentiability theorem for
superposition of mappings
\begin{equation}\label{eq7}
f^{\,\prime}(\gamma^{*}_{il}(s))\gamma^{*\,\prime}_{il}(s)
=e^{i(s+s_i)/r}\,.
\end{equation}
By the relation~(\ref{eq7}), applying the matrix
$(f^{\,\prime}(\gamma^{*}_{il}(s)))^{\,-1}$ to both parts, we will
have that
\begin{equation}\label{eq8}
\gamma^{*\,\prime}_{il}(s)
=(f^{\,\prime}(\gamma^{*}_{il}(s)))^{\,-1}e^{i(s+s_i)/r}\,.
\end{equation}
Therefore, by~(\ref{eq3D}) and~(\ref{eq8}) we obtain that
\begin{equation}\label{eq11}
\int\limits_0^{l(\widetilde{\gamma}_i)}
\widetilde{\rho}(\widetilde{\gamma}^0_i(s))\,ds\geqslant
\int\limits_0^{l(\widetilde{\gamma}_i)}\rho(\gamma^{*}_{il}(s))|\gamma^{*\,\prime}_{il}(
s) | \,ds\,.
\end{equation}
In this case, by the relation~(\ref{eq2A}) it follows that
\begin{equation}\label{eq2B}
\int\limits_{\widetilde{\gamma}_i}{\widetilde{\rho}}(y)\,|dy|
\geqslant
\frac{1}{k}\sum\limits_{l=1}^k\int\limits_0^{l(\widetilde{\gamma}_i)}
\rho(\gamma^{*}_{il}(s))|\gamma^{*\,\prime}_{il}(s)| \,ds\,.
\end{equation}
Let $s_*=s_*(s)$ denote the length of the path $\gamma^{*}_{il}$ on
the segment $[0, s].$ Note that the function $s_*=s_*(s)$ is
absolutely continuous for almost all dashed lines
$\widetilde{\gamma}_i.$ Indeed, let $E\subset [0,
l(\widetilde{\gamma}_i)]$ has a zero Lebesgue measure, i.e.,
$m_1(E)=0.$ Then by~\cite[theorem~3.2.5]{Fe} the set
$|\widetilde{\gamma}_i(E)|=\{y\in {\Bbb C}: \exists\, s\in E:
\widetilde{\gamma}_i(s)=y\}$ has ${\mathcal H}^{1}$-measure is zero.
Since, by the condition of the lemma, the mapping $f$ has
$N^{\,-1}$-property on almost all spheres $S(y_0, r),$ the set
$f^{\,-1}(|\widetilde{\gamma}_i(E))|$ is also of ${\mathcal
H}^{1}$-measure is zero. However, in this case,
by~\cite[theorem~3.2.5]{Fe} the set $E_*:= \{s\in [0,
l(\gamma^{*}_{il})]: \gamma^{*}_ {il}(s)\in
f^{\,-1}(|\widetilde{\gamma}_i(E))\}$ has the linear measure zero.
Note that $E_*=s_*(E),$ therefore, the function $s_*=s_*(s)$ has
Luzin $N$-property. Then, since by the assumption the path
$\gamma^{*}_{il}(s)$ is rectifiable for almost all $r,$ according
to~\cite[Theorem~2.10.13]{Fe} the function $s_*=s_*(s)$ is
absolutely continuous for almost all dashed lines
$\widetilde{\gamma}_i,$ which had to be proved.

\medskip
In this case, the path $\gamma^{\,*}_{il}(s)$ is also absolutely
continuous, because $\gamma^{\,*}_{il}(s)=\gamma^
{\,0}_{il}(s_*(s)),$ and the path $\gamma^{\,0}$ is absolutely
continuous (and even Lipschitz) with respect to its natural
parameter $s_*.$ Then, by~\cite[theorem~3.2.5]{Fe}
\begin{equation}\label{eq9A}
\int\limits_0^{l(\widetilde{\gamma}_i)}
\rho(\gamma^{*}_{il}(s))|\gamma^{*\,\prime}_{il}(s)|
\,ds=\int\limits_{|\gamma^{*}_{il}|}\rho(z)N(\gamma^{*}_{il}, [0,
l(\widetilde{\gamma}_i)], z)\,d{\mathcal H}^{1}(z)\geqslant
\int\limits_{|\gamma^{*}_{il}|}\rho(z)\,d{\mathcal H}^{1}(z)\,.
\end{equation}
Then by~(\ref{eq2B}) and~(\ref{eq9A}) it follows that
\begin{equation}\label{eq10A}
\int\limits_{\widetilde{\gamma}_i}{\widetilde{\rho}}(y)\,|dy|
\geqslant
\frac{1}{k}\sum\limits_{l=1}^k\int\limits_{|\gamma^{*}_{il}|}\rho(z)\,d{\mathcal
H}^{1}(z)\,.
\end{equation}
We sum the ratio~(\ref{eq10A}) over $i=1,2,\ldots .$ Due to the
relation~(3) on the page~\pageref{eq11A} and since $\rho\in
\widetilde{{\rm adm}}\,\Sigma_{\varepsilon},$ we have that
\begin{equation}\label{eq12A}
\int\limits_{S_r}{\widetilde{\rho}}(y)\,|dy|=
\sum\limits_{i=1}^{\infty}
\int\limits_{\widetilde{\gamma}_i}{\widetilde{\rho}}(y)\,|dy|
\geqslant \frac{1}{k}\sum\limits_{i=1}^{\infty}
\sum\limits_{l=1}^k\int\limits_{|\gamma^{*}_{il}|}\rho(z)\,d{\mathcal
H}^{1}(z)\geqslant 1/k
\end{equation}
for $p$-almost all $S_r\in \Gamma,$ that is,
$k\cdot\widetilde{\rho}\in{\rm ext\,adm}_p\,\Gamma.$

\medskip
Recall that the relation~(\ref{eq12A}) has been proved under the
condition that all paths $\gamma^{*}_{il}$ are rectifiable. The
general case, when they are only locally rectifiable, comes
from~(\ref{eq12A}) by taking any rectifiable subpaths
$\widetilde{\widetilde{\gamma}_i}\subset \widetilde{\gamma}_i$
instead of $\widetilde{\gamma}_i.$ As a result of considerations
similar to those carried out above, we will obtain the relation
\begin{equation}\label{eq12B}
\sum\limits_{i=1}^{\infty}
\int\limits_{\widetilde{\widetilde{\gamma}_i}}{\widetilde{\rho}}(y)\,|dy|
\geqslant \frac{1}{k}\sum\limits_{i=1}^{\infty}
\sum\limits_{l=1}^k\int\limits_{|\gamma^{**}_{il}|}\rho(z)\,d{\mathcal
H}^{1}(z)\,,
\end{equation}
where $f\circ\gamma^{**}_{il}={\widetilde{\widetilde{\gamma}_i}}^0,$
$\gamma^{**}_{il}\subset \gamma^{*}_{il}.$ In this case, we are left
just go to $\sup$ over all $\widetilde{\widetilde{\gamma}_i}\subset
\widetilde{\gamma}_i.$ to obtain~(\ref{eq12A}) in the
ratio~(\ref{eq12B}).

\medskip
For the matrix $f^{\,\prime}(z),$ let $\lambda_1(z)\leqslant
\lambda_2(z)$ be the so-called principal numbers, i.e.,
$f^{\,\prime}(z)e_i=\lambda_i(z)\widetilde{e_i}(z)$ for some
orthonormal systems of vectors $e_1, e_2$ and
$\widetilde{e_1},\widetilde{e_2}$ (see, e.g., \cite[Lemma~4.2,
$\S\,4,$ Ch.~I]{Re}). Then $|J(z, f)|=\lambda_1(z)\cdot\lambda_2(z)$
(see relation~(4.5) again), so that $\lambda_1(z)>0$ for each $z\in
D\setminus (B_0\cup B_*).$ In this case,
$\frac{1}{\lambda_2(z)}\leqslant\frac{1}{\lambda_1(z)}$ are the main
numbers for the inverse mapping $(f^{\,\prime}(z))^{\,-1}.$ Since
$m(f(B_0)\cup f(B_*))=0,$ for almost all $y\in f(B_k)$ we have that
\begin{equation}\label{eq12C}
\left(\biggl|(f^{\,\prime}(f^{\,-1}_k(y)))^{\,-1}\frac{y-y_0}{|y-y_0|
}\biggr|\right) \geqslant \frac{1}{\lambda_2(f^{\,-1}_k(y))}>0\,.
\end{equation}
Since $\tilde{\rho}^p(y)=\sup\limits_{k\in {\Bbb
N}}\tilde{\rho}^p_k(y)\leqslant\sum\limits_{k=1}^{\infty}\tilde{\rho}^p_k(y)$
and $m(f(B_*))=m(f(B_0))=0,$ by~(\ref{eq12C}) we obtain that
$$\int\limits_{f(D)}\frac{\tilde{\rho}^p(y)}{Q(y)}\,dm(y)\leqslant
\sum\limits_{k=1}^{\infty}\int\limits_{f(B_k)}
\frac{\tilde{\rho}^p_k(y)}{Q(y)}\,dm(y)\leqslant$$
\begin{equation}\label{eq13A}
\leqslant \sum\limits_{k=1}^{\infty}\int\limits_{f(B_k)}
\frac{\rho^p(f^{\,-1}_k(y))\left(
\biggl|(f^{\,\prime}(f^{\,-1}_k(y)))^{\,-1}\frac{y-y_0}{|y-y_0|}\biggr
|\right) ^{p}|J(y, f^{\,-1}_k)|}{\left(
\biggl|(f^{\,\prime}(f^{\,-1}_k(y)))^{\,-1}\frac{y-y_0}{|y-y_0|}\biggr
|\right) ^{p}}\,dm(y)\,=\end{equation}
$$=\sum\limits_{k=1}^{\infty}\int\limits_{f(B_k)}
\rho^p(f^{\,-1}_k(y))|J(y, f^{\,-1}_k)|\,dm(y)\,.$$
Using the change of variables on each $B_k$, $k=1,2,\ldots$, see,
e.g., \cite[theorem~3.2.5]{Fe}, we obtain that
\begin{equation}\label{eq14B}
\int\limits_{f(B_k)} \rho^p(f^{\,-1}_k(y))|J(y,
f^{\,-1}_k)|\,dm(y)=\int\limits_{B_k} \rho^p(z)\,dm(z)\,.
\end{equation}
It follows by~(\ref{eq13A}) and~(\ref{eq14B}) that
\begin{equation}\label{eq4A}
\int\limits_{f(D)}\frac{{\widetilde{\rho}}^p(y)}{Q(y)}\,dm(y)
\leqslant
\sum\limits_{k=1}^{\infty}\int\limits_{B_k}\rho^p(z)\,dm(z)\,.
\end{equation}
Summing~(\ref{eq4A}) over $k=1,2,\ldots$ and using the countable
additivity of the Lebesgue integral (see, e.g.,
\cite[theorem~I.12.3]{Sa}), we obtain that
\begin{equation}\label{eq5A}
\int\limits_{f(D)}\frac{1}{Q(y)}{\widetilde{\rho}}^p(y)\cdot\,dm(y)
\leqslant\int\limits_D\rho^p(z)\,dm(z)\,.
\end{equation}
Taking $\inf$ in~(\ref{eq5A}) by all functions $\rho \in
\widetilde{{\rm adm}}\,\Sigma_{\varepsilon},$ we obtain that
$$\int\limits_{f(D)}\frac{1}{Q(y)}{\widetilde{\rho}}^p(y)\cdot\,dm(y)
\leqslant\widetilde{M_{p}}(\Sigma_{\varepsilon})\,.$$
Multiplying the last relation on $N^{p}(f, D),$ we obtain that
$$\int\limits_{f(D)}\frac{N^{p}(f, D)}{Q(y)}
{\widetilde{\rho}}^p(y)\cdot\,dm(y) \leqslant N^{p}(f,
D)\cdot\widetilde{M_{p}}(\Sigma_{\varepsilon})\,.$$
Denote by $\widetilde{\rho}_1(y):=N(f, D)\cdot \widetilde{\rho}(y),$
we obtain from the last relation that
\begin{equation}\label{eq10C}
\int\limits_{f(D)} \frac{\widetilde{\rho_1^p(y)}}{Q(y)}\,dm(y)
\leqslant N^{p}(f, D)\cdot\widetilde{M_{p}(\Sigma_{\varepsilon})}\,.
\end{equation}
Since by the proving above $\widetilde{\rho}_1(y)=N(f,
D)\tilde{\rho}\in {\rm ext\,adm}_p\,f(\Sigma_{\varepsilon}),$ it
follows from~(\ref{eq10C}) that~(\ref{eq13}) holds, as required. The
lemma is proved.~$\Box$
\end{proof}

\section{Proof of Theorem~\ref{th1}} Let $Q_*:D\rightarrow
[0,\infty]$ be a Lebesgue measurable function. Denote by
$q_{x_0}(r)$ the integral average of $Q_*(x)$ under the sphere
$|x-x_0|=r,$
\begin{equation}\label{eq3.1A}
q_{x_0}(r):=\frac{1}{2\pi
}\int\limits_0^{2\pi}Q_*(x_0+re^{i\theta})\,d\theta\,.
\end{equation}
Below we also assume that the following standard relations hold:
$a/\infty=0$ for $a\ne\infty,$ $a/0=\infty$ for $a> 0$ and $0\cdot
\infty=0$ (see, e.g., \cite[$\S\,3,$ section~I]{Sa}). The following
conclusion was obtained by V.~Ryazanov together with the author in
the case $p=2,$ see, e.g.,~\cite[Lemma~7.4]{MRSY} or
\cite[Lemma~2.2]{RS}. In the case of an arbitrary $p> 1,$ see, for
example, \cite[Lemma~2]{SalSev$_2$}.

\medskip
\begin{proposition}\label{pr3}
{\sl\, Let $p>1,$ $n\geqslant 2,$ $x_0 \in {\Bbb C},$ $r_1, r_2\in
{\Bbb R},$ $r_1, r_2>0,$ and let $Q_*(x)$ be a Lebesgue measurable
function, $Q_*:{\Bbb C}\rightarrow [0, \infty],$ $Q_*\in L_{\rm
loc}^1({\Bbb C}).$ We put
$$
I=I(x_0,r_1,r_2)=\int\limits_{r_1}^{r_2}
\frac{dr}{r^{\frac{1}{p-1}}q_{x_0}^{\frac{1}{p-1}}(r)}\,,
$$
and let $q_{x_0}(r)$ be defined by~(\ref{eq3.1A}). Then
\begin{equation}\label{eq10D}
\frac{2\pi}{I^{p-1}}\leqslant \int\limits_A Q_*(x)\cdot
\eta^p(|x-x_0|)\,dm(x)
\end{equation}
for any Lebesgue measurable function $\eta :(r_1,r_2)\rightarrow
[0,\infty]$ such that
\begin{equation}\label{eq10B} \int\limits_{r_1}^{r_2}\eta(r)\,dr=1\,,
\end{equation}
where $A=A(x_0, r_1,r_2)$ is defined in~(\ref{eq1**A}).}
\end{proposition}

\medskip
\begin{remark}\label{rem3}
Note that, if~(\ref{eq10D}) holds for any function $\eta$ with a
condition (\ref{eq10B}), then the same relationship holds for any
function $\eta$ with the condition~(\ref{eqA2}). Indeed, let $\eta$
be a nonnegative Lebesgue function that satisfies the condition
(\ref{eqA2}). If $J:=\int\limits_{r_1}^{r_2}\eta(t)\,dt<\infty,$
then we put $\eta_0:=\eta/J.$ Obviously, the function $\eta_0$
satisfies condition~(\ref{eq10B}). Then the relation~(\ref{eq10D})
gives that
$$\frac{2\pi}{I^{p-1}}\leqslant \frac{1}{J^p}\int\limits_A Q_*(x)\cdot
\eta^p(|x-x_0|)\,dm(x)\leqslant \int\limits_A Q_*(x)\cdot
\eta^p(|x-x_0|)\,dm(x)$$
because $J\geqslant 1.$ Let now $J=\infty.$ Then, by
\cite[Theorem~I.7.4]{Sa}, a function $\eta$ is a limit of a
nondecreasing nonnegative sequence of simple functions $\eta_m,$
$m=1,2,\ldots .$ Set
$J_m:=\int\limits_{r_1}^{r_2}\eta_m(t)\,dt<\infty$ and
$w_m(t):=\eta_m(t)/J_m.$ Then, it follows from~(\ref{eq10B}) that
\begin{equation}\label{eq11B}
\frac{2\pi}{I^{p-1}}\leqslant \frac{1}{J_m^p}\int\limits_A
Q_*(x)\cdot \eta_m^p(|x-x_0|)\,dm(x)\leqslant \int\limits_A
Q_*(x)\cdot \eta_m^p(|x-x_0|)\,dm(x)\,,
\end{equation}
because $J_m\rightarrow J=\infty$ as $m\rightarrow\infty$
(see~\cite[Lemma~I.11.6]{Sa}). Thus, $J_m\geqslant 1$ for
sufficiently large $m\in {\Bbb N}.$ Observe that, a functional
sequence $\psi_m(x)=Q_*(x)\cdot \eta_m^p(|x-x_0|),$ $m=1,2\ldots ,$
is nonnegative, monotone increasing and converges to a function
$\psi(x):=Q_*(x)\cdot \eta^p(|x-x_0|)$ almost everywhere. By the
Lebesgue theorem on the monotone convergence
(see~\cite[Theorem~I.12.6]{Sa}), it is possible to go to the limit
on the right side of the inequality~(\ref{eq11B}), which gives us
the desired inequality~(\ref{eq10D}).
\end{remark}

\medskip
{\it Proof of Theorem~\ref{th1}}. Fix $y_0\in
\overline{f(D)}\setminus\{\infty\},$
$0<r_1<r_2<r_0=\sup\limits_{y\in f(D)}|y-y_0|,$ $C_1\subset B(y_0,
r_1)\cap f(D)$ and $C_2\subset f(D)\setminus B(y_0, r_2).$ Set
$$C_0:=\overline{f^{\,-1}(C_1)}\,,\quad
C^*_0:=\overline{f^{\,-1}(C_2)}$$ (see Figure~\ref{fig1}).
\begin{figure}[h]
\centering\includegraphics[scale=0.55]{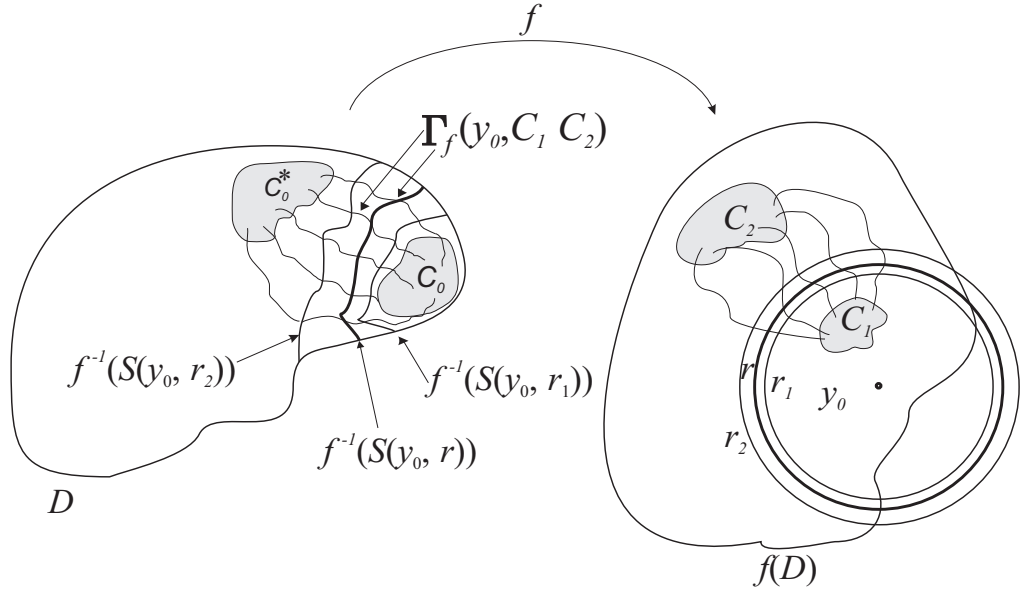}
\caption{To the proof of Theorem~\ref{th1}} \label{fig1}
\end{figure}
Observe that $C_0$ and $C_1$ are disjoint compact sets in $D,$ see
\cite[Theorem~3.3]{Vu}. Besides that, $C_1$ and $C_2$ are non empty
by the choose of $r_0,$ $r_1$ and $r_2.$

\medskip Let us to show that a set $\sigma_r:=f^{\,-1}(S(y_0, r))$
separates $C_0$ from $C^{\,*}_0$ in $D$ for any $r\in (r_1, r_2).$
Indeed, $\sigma_r$ is closed in $D$ as a preimage of a closed set
$S(y_0, r)$ under the continuous mapping~$f$ (see, e.g.,
\cite[Theorem~1.IV.13, Ch.~1]{Ku}). In particular, $\sigma_r$ is
also closed with respect to $R:=D\setminus (C_0\cup C^{\,*}_0).$ We
put
$$A:=f^{\,-1}(B(y_0,
r))$$ and
$$B:=D\setminus \overline{f^{\,-1}(B(y_0, r))}\,.$$
Observe that, $A$ and $B$ are not empty by the choice of $r_0,$
$r_1,$ $r_2$ and $r.$ Since $f$ is continuous, $f^{\,-1}(B(y_0, r))$
and $D\setminus \overline{f^{\,-1}(B(y_0, r))}$ are open in $D.$ In
other words, $A$ and $B$ are open in
$$R^{\,*}:=R\cup C_0\cup C_1=D\,.$$
Note that $A\cap B=\varnothing,$ and $R^{\,*}\setminus
\sigma_r=A\cup B.$ Let $\Sigma_{C_0, C^{\,*}_0}$ be the family of
all sets separating $C_0$ and $C^{\,*}_0$ in $R^{\,*}.$ In this
case, by the equations of Ziemer and Hesse, see (\ref{eq3})
and~(\ref{eq4}), respectively, we obtain that
\begin{equation}\label{eq7A}
M_{\alpha}(\Gamma_f(y_0, C_1,
C_2))=\left(\widetilde{M}_{p}(\Sigma_{r_1,
r_2})\right)^{1-\alpha}\,,
\end{equation}
where $\alpha=\frac{p}{p-1}.$
Then by Lemma~\ref{lem1} and by the relation~(\ref{eq7A}), we obtain
that
\begin{equation}\label{eq8A}
M_{\alpha}(\Gamma_f(y_0, r_1, r_2))\leqslant
\left(\inf\limits_{\rho\in{\rm
ext\,adm}\,f(\Sigma_{\varepsilon})}\int\limits_{f(D)\cap A(y_0, r_1,
r_2)}\frac{\rho^p(y)}{N^{p}(f, D)\cdot
Q(y)}\,dm(y)\right)^{-\frac{1}{p-1}}\,,
\end{equation}
where $Q(y):=K_{CT, p, y_0}(y, f):=\sum\limits_{z\in
f^{\,-1}(y)}\frac{\left(\biggl|(f^{\,\prime}(z))^{\,-1}\frac{y-y_0}{|y-y_0|}\biggr|\right)
^{p}}{|J(z, f)|}.$ Using the second remote formula in the proof of
Theorem~9.2 in \cite{MRSY}, we obtain that
$$\inf\limits_{\rho\in{\rm
ext\,adm}\,f(\Sigma_{\varepsilon})}\int\limits_{f(D)\cap A(y_0, r_1,
r_2)}\frac{\rho^p(y)}{N^{p}(f, D)\cdot Q(y)}\,dm(y)=$$
\begin{equation}\label{eq9}
=\int\limits_{r_1}^{r_2} \left(\inf\limits_{\alpha\in
I(r)}\int\limits_{S(y_0, r)\cap f(D)}\frac{\alpha^{p}(y)}{N^p(f,
D)\cdot Q(y)}\,\,\mathcal{H}^{1}(y)\right)\,dr\,,
\end{equation}
where $I(r)$ denotes the set of all measurable functions on $S(y_0,
r)\cap f(D)$ such that
$$\int\limits_{S(y_0, r)\cap f(D)}\alpha(x)\,\mathcal{H}^{1}=1\,.$$
Then, choosing $X=S(y_0, r)\cap f(D),$ $\mu={\mathcal H}^{1}$ and
$\varphi=\frac{1}{Q}|_{S(y_0, r)\cap f(D)}$
in~\cite[Lemma~9.2]{MRSY}, we obtain that
\begin{equation}\label{eq10}
\int\limits_{r_1}^{r_2} \left(\inf\limits_{\alpha\in
I(r)}\int\limits_{S(y_0, r)\cap
f(D)}\frac{\alpha^p(y)}{Q(y)}\,\,d\mathcal{H}^{1}\right)\,dr=
\int\limits_{r_1}^{r_2}\frac{dr}{\Vert Q\Vert_s(r)}\,,
\end{equation}
where $\Vert Q\Vert_s(r)=\left(\int\limits_{S(y_0, r)\cap
f(D)}Q^s(x)\,d{\mathcal H}^{1}\right)^{1/s}$ and
$s:=\frac{1}{p-1}=\alpha-1.$ Thus, by~(\ref{eq8A}), (\ref{eq9}) and
(\ref{eq10}) we obtain that
$$M_{\alpha}(\Gamma_f(y_0, r_1, r_2))\leqslant N^{\alpha}(f, D)\cdot
\left(\int\limits_{r_1}^{r_2}\frac{dr}{\Vert
Q\Vert_s(r)}\right)^{-\frac{1}{p-1}}=$$
\begin{equation}\label{eq11C}
=\frac{{N^{\alpha}}(f, D)\cdot
2\pi}{\left(\int\limits_{r_1}^{r_2}\frac{dr}{r^{\frac{1}{\alpha-1}}
\widetilde{q}^{1/(\alpha-1)}_{y_0}(r)}\right)
^{\frac{1}{{p-1}}}}=\frac{N^{\alpha}(f, D)\cdot
2\pi}{\left(\int\limits_{r_1}^{r_2}\frac{dr}{r^{\frac{1}{\alpha-1}}
\widetilde{q}^{1/(\alpha-1)}_{y_0}(r)}\right)^{\alpha-1}}\,,
\end{equation}
where $q_{y_0}(r)=\frac{1}{2\pi r}\int\limits_{S(y_0,
r)}\widetilde{Q}\,d\mathcal{H}^{1}$ and
$\widetilde{Q}(y)=\begin{cases}
Q^{\alpha-1}(y)\,,&y\in f(D)\,,\\
0\,,&y\not\in f(D)\end{cases}.$
Finally, it follows from~(\ref{eq11C}) and Proposition~\ref{pr3}
that the relation
$$M_{\alpha}(\Gamma_f(y_0, r_1, r_2))\leqslant \int\limits_{A(y_0,r_1,r_2)\cap
f(D)}N^{\alpha}(f, D)\cdot Q^{\alpha-1}(y)\cdot \eta^{\,\alpha}
(|y-y_0|)\, dm(y)$$
holds for a function $Q(y):=K_{CT, \frac{\alpha}{\alpha-1}, y_0}(y,
f):=\sum\limits_{z\in
f^{\,-1}(y)}\frac{\left(\biggl|(f^{\,\prime}(z))^{\,-1}\frac{y-y_0}{|y-y_0|}\biggr|\right)
^{\frac{\alpha}{\alpha-1}}}{|J(z, f)|}\,,$ that is desired
conclusion.~$\Box$

\section{Proofs of Theorems~\ref{th1B}--\ref{th2}}

\medskip
The next class of mappings is a generalization of quasiconformal
mappings in the sense of Gehring's ring definition (see \cite{Ge};
it is the subject of a separate study, see, e.g.,
\cite[Chapter~9]{MRSY}). Let $D$ and $D^{\,\prime}$ be domains in
${\Bbb C}.$ Suppose that $x_0\in\overline {D}\setminus\{\infty\}$
and $Q\colon D\rightarrow(0,\infty)$ is a Lebesgue measurable
function. A function $f\colon D\rightarrow D^{\,\prime}$ is called a
{\it lower $Q$-mapping at a point $x_0$ relative to the $p$-modulus}
if
\begin{equation}\label{eq1A}
M_p(f(\Sigma_{\varepsilon}))\geqslant \inf_{\rho\in{\rm ext}_p\,{\rm
adm}\Sigma_{\varepsilon}}\int\limits_{D\cap A(x_0, \varepsilon,
r_0)}\frac{\rho^p(x)}{Q(x)}\,dm(x)
\end{equation}
for every spherical ring $A(x_0, \varepsilon, r_0)=\{x\in {\Bbb
C}\,:\, \varepsilon<|x-x_0|<r_0\}$, $r_0\in(0,d_0)$, $d_0=\sup_{x\in
D}|x-x_0|$, where $\Sigma_{\varepsilon}$ is the family of all
intersections of the spheres $S(x_0, r)$ with the domain $D$, $r\in
(\varepsilon, r_0)$. If $p=2$, we say that $f$ is a lower
$Q$-mapping at $x_0$. We say that $f$ is a lower $Q$-mapping
relative to the $p$-modulus in $A\subset \overline {D}$ if
(\ref{eq1A}) is true for all $x_0\in A$.

\medskip
The following statement can be proved much as Theorem 9.2 in \cite{MRSY}, so we omit the arguments.

\begin{lemma}\label{lem4}{\sl\, Let $D$,
$D^{\,\prime}\subset\overline{{\Bbb C}},$ let $x_0\in\overline
{D}\setminus\{\infty\}$, and let $Q$ be a Lebesgue measurable
function. A mapping $f\colon D\rightarrow D^{\,\prime}$ is a lower
$Q$-mapping relative to the $p$-modulus at a point $x_0$, $p>1$, if
and only if $M_p(f(\Sigma_{\varepsilon}))\geqslant
\int\limits_{\varepsilon}^{r_0} \frac{dr}{\|\,Q\|_{s}(r)}$ for all
$\varepsilon\in(0,r_0),\ r_0\in(0,d_0)$, $d_0=\sup_{x\in D}|x-x_0|$,
$s=\frac{1}{p-1}$, where, as above, $\Sigma_{\varepsilon}$ denotes
the family of all intersections of the spheres $S(x_0, r)$ with $D$,
$r\in (\varepsilon, r_0)$, $\|
Q\|_{s}(r)=(\int\limits_{D(x_0,r)}Q^{s}(x)\,d{\mathcal{A}})^{\frac{1}{s}}$
is the $L_{s}$-norm of $Q$ over the set ${D(x_0,r)=\{x\in D\,:\,
|x-x_0|=r\}=D\cap S(x_0,r)}\,.$}
\end{lemma}

\medskip
The following statement holds, cf.~\cite[Theorem~2.1]{KR},
\cite[Lemma~2.3]{SSP}, \cite[Lemma~2]{Sev$_2$}
and~\cite[Theorem~4.1]{RSY$_2$}.

\begin{theorem}\label{thOS4.2}{\sl\, Let $p>1$ and let
$f:D\rightarrow {\Bbb C}$ be an open discrete mapping of a finite
distortion such that $N(f, D)<\infty.$ Then $f$ satisfies the
relation~(\ref{eq1A}) at any $z_0\in\overline{D}$ for $Q(z)=N(f,
D)\cdot K_{T, p, z_0}(z, f),$ where
$$Q(z):=K_{T, p, z_0}(z,
f):=\frac{\left(\biggl|(f^{\,\prime}(z))\frac{z-z_0}{|z-z_0|}\biggr|\right)
^{p}}{|J(z, f)|}$$
and $N(f, D)$ is defined in~(\ref{eq15}).}
\end{theorem}

\begin{proof}
The proof of this Theorem uses the scheme outlined
in~\cite[Theorem~4]{Sev$_1$}, cf. \cite[Theorem~4.1]{RSY$_2$}.
Observe that, $f=\varphi\circ g,$ where $g$ is some homeomorphism
and $\varphi$ is an analytic function, see (\cite[5.III.V]{St}).
Thus, $f$ is differentiable almost everywhere (see, e.g.,
\cite[Theorem~III.3.1]{LV}). Let $B$ be a Borel set of all points
$z\in D,$ where $f$ has a total differential $f^{\,\prime}(z)$ and
$J(z, f)\ne 0$. Observe that, $B$ may be represented as a countable
union of Borel sets $B_l$, $l=1,2,\ldots\,,$ such that
$f_l=f|_{B_l}$ are bi-lipschitzian homeomorphisms (see
\cite[items~3.2.2, 3.1.4 and 3.1.8]{Fe}). Without loss of
generality, we may assume that the sets $B_l$ are pairwise disjoint.
Denote by $B_*$ the set of all points $z\in D$ in which $f$ has a
total differential, however, $f^{\,\prime}(z)=0.$

Since $f$ is of finite distortion, $f^{\,\prime}(z)=0$ for almost
all $z,$ where $J(z, f)=0.$ Thus, by the construction the set
$B_0:=D\setminus \left(B\bigcup B_*\right)$ has a zero Lebesgue
measure. Therefore, by Proposition~\ref{pr2} and Remark~\ref{rem2}
$l(B_0\cap S_r)=0$ for $p$-almost all circles $S_r:=S(z_0,r)$
centered at $z_0\in\overline{D}.$ Observe that, a function
$\psi(r):={\mathcal H}^{\,1}(B_0\cap S_r)=l(B_0\cap S_r)$ is
Lebesgue measurable by the Fubini theorem, thus, the set $E=\{r\in
{\Bbb R}: l(B_0\cap S_r)=0\}$ is Lebesgue measurable. Now, by
Proposition~\ref{pr7} we obtain that
\begin{equation}\label{eq16A}
l(B_0\cap S_r)=0 \quad \text{for\,\,almost\,\, any\,\,} r\in {\Bbb
R}\,.
\end{equation}
Let us fix the circle
$$\gamma(t)=S(z_0,r)=re^{it}+y_0, \quad t\in [0, 2\pi)\,.$$
Then $S_r:=S(y_0,r)\cap D$ is a family of dashed lines
$$\gamma_i: \bigcup\limits_{i=1}^{\infty} (a_i,
b_i)\rightarrow D,\quad (a_i, b_i)\subset {\Bbb R},\quad
i=1,2,\ldots \,.$$
Consider the normal representation $\gamma^0_i$ for $\gamma_i,$
namely,
\begin{equation}\label{eq14A}
\gamma^0_i(t)=z_0+re^{i(t+t_0)/r}\,,\qquad t\in (0,
l(\gamma^0_i))\,,
\end{equation}
where $t_0\in {\Bbb R}$ is a number such that
$\gamma^0_i(t_0)=\gamma_i(a_i),$ and the number $l(\gamma^0_i)$
denotes the length of $\gamma_i$ (see \cite[Definition~2.5]{Va}).

\medskip
Let $\Gamma$ be a family of all intersections of circles $S_r$,
$r\in(\varepsilon,\varepsilon_0),$
$\varepsilon_0<d_0=\sup\limits_{z\in D}\,|z-z_0|,$ with a domain
$D.$ Given an admissible function $\rho_*\in{\rm adm}\,f(\Gamma),$
$\rho_*\equiv 0$ outside of $f(D),$ we set $\rho\equiv 0$ outside of
$D$ and on $B_0\cup B_*,$ and
$$\rho(z)\colon=\rho_*(f(z))
\biggl|(f^{\,\prime}(z))\frac{z-z_0}{|z-z_0|}\biggr| \qquad\text{for}\ z\in D\setminus
B_0\,.$$
Denote by
$$\gamma_i(B_0\cup B_*)=\{t\in (0, l(\gamma_i)): \gamma_i(t)\in B_0\cup B_*\}\,.$$
%
%
%
Using~(\ref{eq14A}) and theorem on the derivative of superpositions
of mappings, we obtain that
$$\int\limits_{\gamma_i}
\rho(z)\,|dz|=\int\limits_{[0, l(\gamma_i)]\setminus
\gamma_i(B_0\cup B_*)} \rho(\gamma_i^0(t))\,dt+\int\limits_{
\gamma_i(B_0\cup B_*)} \rho(\gamma_i^0(t))\,dt=$$$$=\int\limits_{[0,
l(\gamma_i)]\setminus \gamma_i(B_0\cup B_*)}
\rho_*(f(z_0+re^{i(t+t_0)/r}))
\biggl|(f^{\,\prime}(z_0+re^{i(t+t_0)/r}))re^{i(t+t_0)/r}\biggr|\,dt=$$
\begin{equation}\label{eq15A}
=\int\limits_{[0, l(\gamma_i)]\setminus \gamma_i(B_0\cup B_*)}
\rho_*(f(z_0+re^{i(t+t_0)/r}))
\biggl|\left(f(z_0+re^{i(t+t_0)/r})(t)\right)^{\prime}\biggr|\,dt=
\end{equation}
$$=\int\limits_0^{l(\gamma_i)}
\rho_*(f(z_0+re^{i(t+t_0)/r}))
\biggl|
\left(f(z_0+re^{i(t+t_0)/r})(t)\right)^{\prime}\biggr|\,dt=\int\limits_{f(\gamma_i)}
\rho_*(w)\,|dw|$$
for a.e. $r,$ because $f\in W_{loc}^{1, 1}$ and, consequently, the
path $f(\gamma_i(t))$ is absolutely continuous over $t$ for a.e. $r$
(see~\cite[Theorem~1.I]{Ma} and~\cite[Theorem~4.1]{Va}). Here we
have used the fact this $z_0+re^{i(t+t_0)/r}\not\in B_0$ for a.e.
$r$ and for a.e. $t,$ that follows from~(\ref{eq16A}). Similarly,
since by the condition $f$ is a mapping with a finite distortion,
that implies that
$\left(f(z_0+re^{i(t+t_0)/r})(t)\right)^{\prime}=0$ for any $t$ such
that $z_0+re^{i(t+t_0)/r}\in B_*,$ that has been used
in~(\ref{eq15A}), as well. The relation~(\ref{eq15A}) implies that
$$\int\limits_{\gamma_i}
\rho(z)\,|dz|=\int\limits_{f(\gamma_i)} \rho_*(w)\,|dw|$$
for a.e. $r.$ Summing the last relation over all $i=1,2,\ldots ,$ we
obtain that
$$\int\limits_{S_r}
\rho(z)\,|dz|=\int\limits_{f(S_r)} \rho_*(w)\,|dw|\geqslant 1$$
for a.e. $r$ and, consequently, for $p$-a.e. $S_r$ in the sense of
$p$-modulus (see Proposition~\ref{pr7}). Thus,
$\rho\in{\rm{ext}}_p{\rm\,adm}\,\Gamma.$

\medskip
Using the change of variables on $B_l$, $l=1,2,\ldots$ (see, e.g.,
\cite[Theorem~3.2.5]{Fe}), by the countable additivity of the
Lebesgue integral we obtain that
$$\int\limits_{D}\frac{\rho^p(z)}{K_{T, p, z_0}(z,
f)}\,dm(z)=\sum\limits_{l=1}^{\infty}\int\limits_{B_l}
\rho^p(f(z))|J(z, f)|\,dm(z)=$$
$$=\sum\limits_{l=1}^{\infty}\int\limits_{f(B_l)}\rho^p_*(y)\,dm(y)
\leqslant \int\limits_{f(D)}N(f, D)\rho^p_*(y)\,dm(y)\,,$$ as
required.~$\Box$
\end{proof}

We also need the following statement given in
\cite[Proposition~10.2, Ch.~II]{Ri}.

\begin{proposition}\label{pr4}{\sl\,
Let $E=(A,\,C)$ be a condenser in ${\Bbb C}$ and let $\Gamma_E$ be
the family of all paths of the form $\gamma:[a,\,b)\rightarrow A$
with $\gamma(a)\in C$ and $|\gamma|\cap(A\setminus F)\ne\varnothing$
for every compact $F\subset A.$ Then ${\rm cap}_q\,E=
M_q(\Gamma_E).$}
\end{proposition}

\medskip
An analogue of the following assertion has been proved several times
earlier under slightly different conditions,
see~\cite[Lemma~4.2]{SevSkv} and \cite[Lemma~5]{Sev$_1$}. In the
formulation given below, this result is proved for the first time.

\medskip
 \begin{lemma}\label{l4.4}
{\,\sl Let $D$ be a domain in ${\Bbb C},$ let $p>1,$ let $x_0\in D$
and let $f:D\rightarrow {\Bbb C}$ be an open and discrete mapping
satisfying the relation~(\ref{eq1A}) at a point $x_0.$ Assume that
$Q\colon D\rightarrow[0,\infty]$ is a Lebesgue measurable function
which is locally integrable in the degree $s=\frac{1}{p-1}$ in $D.$
Then the relation
\begin{equation}\label{eq4B}
{\rm cap}_{\alpha}\, f(\mathcal{E})\leqslant\int\limits_{A}
Q^{\,*}(x)\cdot \eta^{\alpha}(|x-x_0|)\, dm(x)
\end{equation}
holds for $\alpha=\frac{p}{p-1}$ and
$Q^{\,*}(x)=Q^{\frac{1}{p-1}}(x),$ where $\mathcal{E}=(B(x_0, r_2),
\overline{B(x_0, r_1)}),$ $A=A(x_0, r_1, r_2),$
$0<r_1<r_2<\varepsilon_0:={\rm dist}\,(x_0,
\partial D),$ and $\eta \colon (r_1,r_2)\rightarrow [0,\infty ]$ may
be chosen as arbitrary nonnegative Lebesgue measurable function
satisfying the relation~(\ref{eq8BC}).}
 \end{lemma}

\begin{proof} Observe that $s=\alpha-1.$ By Lemma~2
in~\cite{SalSev$_2$}, it is sufficiently to prove that
$${\rm cap_{\alpha}}\,
f(\mathcal{E})\leqslant \frac{2\pi}{I^{*\,\alpha-1}}\,,$$
where $\mathcal{E}$ is a condenser $\mathcal{E}=(B(x_0, r_2),
\overline{B(x_0, r_1)}),$ and $q^{\,*}_{x_0}(r)$ denotes the
integral average of $Q^{\alpha-1}(x)$ under $S(x_0, r),$
\begin{equation}\label{eq16}
q_{x_0}(r)=\frac{1}{2\pi r}\int\limits_{S(x_0,
r)}Q(x)\,d\mathcal{H}^{1}\,,
\end{equation}
where $$I^{\,*}=I^{\,*}(x_0, r_1,r_2)=\int\limits_{r_1}^{r_2}\
\frac{dr}{r^{\frac{1}{\alpha-1}}q^{\,*\,\frac{1}{\alpha-1}}_{x_0}(r)}\,.$$
Let $\varepsilon\in (r_1, r_2)$ and let $B(x_0, \varepsilon).$ We
put $C_0=\partial f(B(x_0, r_2)),$ $C_1=f(\overline{B(x_0, r_1)}),$
$\sigma=\partial f(B(x_0, \varepsilon)).$ Since $f$ is continuous in
$D,$ the set $f(B(x_0, r_2))$ is bounded.

\medskip
Since $f$ is continuous, $\overline{f(B(x_0, r_1))}$ is a compact
subset of $f(B(x_0, \varepsilon)),$ and $\overline{f(B(x_0,
\varepsilon))}$ is a compact subset of $f(B(x_0, r_2)).$ In
particular,
$$\overline{f(B(x_0, r_1))}\cap
\partial f(B(x_0, \varepsilon))=\varnothing\,.$$
Let, as above, $R=G \setminus (C_{0} \cup C_{1}),$ $G:=f(D),$ and
$R^{\,*} = R \cup C_{0}\cup C_{1}.$ Then $R^{\,*}.$ Observe that,
$\sigma$ separates $C_0$ from $C_1$ in $R^{\,*}=G.$ Indeed, the set
$\sigma \cap R$ is closed in $R,$ besides that, if $A:=G\setminus
\overline{f(B(x_0, \varepsilon))}$ and $B= f(B(x_0, \varepsilon)),$
then $A$ and $B$ are open in $G\setminus \sigma,$ $C_0\subset A,$
$C_1\subset B$ and $G\setminus \sigma=A\cup B.$

\medskip
Let $\Sigma$ be a family of all sets, which separate $C_0$ from
$C_1$ in $G.$ Below by $\bigcup\limits_{r_1<r<r_2}
\partial f(B(x_0, r))$ or $\bigcup\limits_{r_1<r<r_2}
f(S(x_0, r))$ we mean the union of all Borel sets into a family, but
not in a theoretical-set sense (see \cite[item~3, p.~464]{Zi$_1$}).
Let $\rho\in \widetilde{{\rm adm}}\bigcup\limits_{r_1<r<r_2}
\partial f(B(x_0, r))$ in the sense of the relation~(\ref{eq13.4.13}). Then $\rho\in {\rm
adm}\bigcup\limits_{r_1<r<r_2}
\partial f(B(x_0, r))$ in the sense of~(\ref{eq8.2.6}).
By the openness of the mapping $f$ we obtain that  $\partial
f(B(x_0, r))\subset f(S(x_0, r)),$  therefore, $\rho\in {\rm
adm}\bigcup\limits_{r_1<r<r_2} f(S(x_0, r))$ and, consequently,
by~(\ref{eq13.4.12})
\begin{multline}\label{eq5E}
\widetilde{M_p}(\Sigma)\geqslant
\widetilde{M_p}\left(\bigcup\limits_{r_1<r<r_2}
\partial f(B(x_0, r))\right)\geqslant\\
\geqslant \widetilde{M_p}\left(\bigcup\limits_{r_1<r<r_2}
f(S(x_0, r))\right)\geqslant\\
\geqslant M_p\left(\bigcup\limits_{r_1<r<r_2} f(S(x_0, r))\right).
\end{multline}
However, by~(\ref{eq3}) and ~(\ref{eq4}) we obtain that
\begin{equation}\label{eq6D}
\frac{1}{(M_{\alpha}(\Gamma(C_0, C_1, G)))^{1/(\alpha-1)}}=
\widetilde{M_{p}}(\Sigma)\,.
\end{equation}
Let $\Gamma_{f(\mathcal{E})}$ be a family of all paths which
correspond to the condenser $f(\mathcal{E})$ in the sense of
Proposition~\ref{pr4}, and let $\Gamma^{\,*}_{f(\mathcal{E})}$ be a
family of all rectifiable paths of $\Gamma_{f(\mathcal{E})}.$ Now,
observe that, the families $\Gamma^{*}_{f(\mathcal{E})}$ and $\Gamma
(C_0, C_1, G)$ have the same families of admissible functions
$\rho.$ Thus,
$$M_{\alpha}(\Gamma_{f(\mathcal{E})})=M_{\alpha}(\Gamma(C_0, C_1, G))\,.$$
By Proposition~\ref{pr4}, we obtain that
$M_{\alpha}(\Gamma_{f(\mathcal{E})})={\rm
cap}_{\alpha}f(\mathcal{E}).$ By~(\ref{eq6D}) we obtain that
\begin{equation}\label{eq7B}
\left(\widetilde{M_p}(\Sigma)\right)^{\alpha-1}=\frac{1}{{\rm
cap}_{\alpha}f(\mathcal{E})}\,.
\end{equation}
Finally, by~(\ref{eq5E}) and~(\ref{eq7B}) we obtain that
$$
{\rm cap}_{\alpha}f(\mathcal{E}) \leqslant \frac{1}{M_{\alpha}
\left(\bigcup\limits_{r_1<r<r_2} f(S(x_0, r))\right)^{\alpha-1}}\,.
$$
By Lemma~\ref{lem4}, we obtain that
$$
{\rm cap}_{\alpha}f(\mathcal{\mathcal{E}}) \leqslant
\frac{1}{\left(\int\limits_{r_1}^{r_2} \frac{dr}{\Vert
\,Q\Vert_{s}(r)}\right)^{s}}=\frac{1}{I^{*\,\alpha-1}}\,,
$$
as required.~$\Box$
\end{proof}

\medskip
The following result is proved in~\cite[Theorem~5]{Sev$_3$}.

\medskip
\begin{proposition}\label{pr5}
{\,\sl Let $x_0\in \partial D,$ let $f:D\rightarrow {\Bbb C}$ be an
open, discrete and closed bounded lower $Q$-mapping with a respect
to $p$-modulus in $D\subset{\Bbb C},$ $Q\in
L_{loc}^{\frac{1}{p-1}}({\Bbb C}),$ $1<p,$ and
$\alpha:=\frac{p}{p-1}.$ Then for any
$\varepsilon_0<d_0:=\sup\limits_{x\in D}|x-x_0|$ and any compactum
$C_2\subset D\setminus B(x_0, \varepsilon_0)$ there is
$\varepsilon_1,$ $0<\varepsilon_1<\varepsilon_0,$ such that the
inequality
\begin{equation}\label{eq3B}
M_{\alpha}(f(\Gamma(C_1, C_2, D)))\leqslant \int\limits_{A(x_0,
\varepsilon, \varepsilon_1)}Q^{\frac{1}{p-1}}(x)
\eta^{\alpha}(|x-x_0|)\,dm(x)\,,
\end{equation}
holds for any $\varepsilon\in (0, \varepsilon_1)$ and any compactum
$C_1\subset \overline{B(x_0, \varepsilon)}\cap D,$ where $A(x_0,
\varepsilon, \varepsilon_1)=\{x\in {\Bbb C}:
\varepsilon<|x-x_0|<\varepsilon_1\}$ and $\eta: (\varepsilon,
\varepsilon_1)\rightarrow [0,\infty]$ is any nonnegative Lebesgue
measurable function such that
\begin{equation}\label{eq6B}
\int\limits_{\varepsilon}^{\varepsilon_1}\eta(r)\,dr=1\,.
\end{equation}
 }
\end{proposition}

\begin{remark}
Note that, if~(\ref{eq2*!A}) holds for any function $\eta$ with a
condition (\ref{eq6B}), then the same relationship holds for any
function $\eta$ with the condition~(\ref{eq8BC}). Indeed, let $\eta$
be a nonnegative Lebesgue function that satisfies the condition
(\ref{eq8BC}). If $J:=\int\limits_{r_1}^{r_2}\eta(t)\,dt<\infty,$
then we put $\eta_0:=\eta/J.$ Obviously, the function $\eta_0$
satisfies condition~(\ref{eq6B}). Then the relation~(\ref{eq3B})
gives that
$$M_{\alpha}(f(\Gamma(C_1, C_2, D)))\leqslant
$$
\begin{equation}\label{eq1X}
\frac{1}{J^{\alpha}}\int\limits_A Q(x)\cdot
\eta^{\alpha}(|x-x_0|)\,dm(x)\leqslant \int\limits_A Q(x)\cdot
\eta^{\alpha}(|x-x_0|)\,dm(x)
\end{equation}
because $J\geqslant 1.$ Let now $J=\infty.$ Then, by
\cite[Theorem~I.7.4]{Sa}, a function $\eta$ is a limit of a
nondecreasing nonnegative sequence of simple functions $\eta_m,$
$m=1,2,\ldots .$ Set
$J_m:=\int\limits_{r_1}^{r_2}\eta_m(t)\,dt<\infty$ and
$w_m(t):=\eta_m(t)/J_m.$ Then, similarly to~(\ref{eq1X}) we obtain
that
$$M_{\alpha}(f(\Gamma(C_1, C_2, D)))\leqslant
$$
\begin{equation}\label{eq11E}
\frac{1}{J_m^{\alpha}}\int\limits_A Q(x)\cdot
\eta_m^{\alpha}(|x-x_0|)\,dm(x)\leqslant \int\limits_A Q(x)\cdot
\eta_m^{\alpha}(|x-x_0|)\,dm(x)\,,
\end{equation}
because $J_m\rightarrow J=\infty$ as $m\rightarrow\infty$
(see~\cite[Lemma~I.11.6]{Sa}). Thus, $J_m\geqslant 1$ for
sufficiently large $m\in {\Bbb N}.$ Observe that, a functional
sequence $\varphi_m(x)=Q_*(x)\cdot \eta_m^{\alpha}(|x-x_0|),$
$m=1,2\ldots ,$ is nonnegative, monotone increasing and converges to
a function $\varphi(x):=Q_*(x)\cdot \eta^{\alpha}(|x-x_0|)$ almost
everywhere. By the Lebesgue theorem on the monotone convergence
(see~\cite[Theorem~I.12.6]{Sa}), it is possible to go to the limit
on the right side of the inequality~(\ref{eq11E}), which gives us
the desired inequality~(\ref{eq2*!A}).
\end{remark}

\medskip
The following result is proved in~\cite[Theorem~6]{Sev$_2$}.

\medskip
\begin{proposition}\label{pr6}
{\,\sl Let $x_0\in \partial D,$ and let $f:D\rightarrow {\Bbb C}$ be
a bounded lower $Q$-homeomorphism with respect to $p$-modulus in a
domain $D\subset{\Bbb C},$ $Q\in L_{\rm loc}^{\frac{1}{p-1}}({\Bbb
C}),$ $p>1$ and $\alpha:=\frac{p}{p-1}.$ Then $f$ is a ring
$Q^{\frac{1}{p-1}}$-homeomorphism  with respect to $\alpha$-modulus
at this point, where $\alpha:=\frac{p}{p-1}.$
 }
\end{proposition}

\medskip
{\it Proof of Theorem~\ref{th1B}.} Fix $z_0\in \overline{D}.$ Two
situations are possible: when $z_0\in D,$ and when $z_0\in \partial
D.$ Let $z_0\in D.$ Set $p:=\frac{\alpha}{\alpha-1},$ then
$\alpha=\frac{p}{p-1}.$ Due to Theorem~\ref{thOS4.2}, $f$ satisfies
the relation~(\ref{eq1A}) with
\begin{equation}\label{eq17A} Q(z):=N(f, D)\cdot K_{T, p, z_0}(z, f)=N(f,
D)\cdot\frac{\left(\biggl|(f^{\,\prime}(z))\frac{z-z_0}{|z-z_0|}\biggr|\right)
^{p}}{|J(z, f)|}\,.
\end{equation}
Now, by Lemma~\ref{l4.4} $f$ satisfies the relation~(\ref{eq4B})
with
\begin{equation}\label{eq18}
Q^{\,*}(z):= K^{\alpha-1}_{T, p, z_0}(z,
f):=\left(\frac{\left(\biggl|(f^{\,\prime}(z))\frac{z-z_0}{|z-z_0|}\biggr|\right)
^{p}}{|J(z, f)|}\right)^{\alpha-1}\,.
\end{equation}
Observe that the relation
\begin{equation}\label{eq5}
M_{\alpha}(f(\Gamma(C_1, C_2, D)))\leqslant {\rm
cap}_{\alpha}(f(B(z_0, r_2)), f(\overline{B(z_0, r_1)}))
\end{equation}
holds for all $0<r_1<r_2<d_0:={\rm dist}\, (z_0,
\partial D)$ and for any
continua $C_1\subset \overline{B(z_0, r_1)},$ $C_2\subset D\setminus
B(z_0, r_2).$ Indeed, $f(\Gamma(C_1, C_2, D))>\Gamma_{f(E)},$ where
$E:=(f(B(z_0, r_2)), f(\overline{B(z_0, r_1)})),$ and
$\Gamma_{f(E)}$ is a family from Proposition~\ref{pr4} for the
condenser $f(E).$ The relation~(\ref{eq5}) finishes the proof for
the case $z_0\in D.$ Let now $z_0\in \partial{D}.$ Again, by
Theorem~\ref{thOS4.2}, $f$ satisfies the relation~(\ref{eq1A}) with
$Q(z)$ from~(\ref{eq17A}). Now $f$ satisfies the
relation~(\ref{eq2*!A}) with $Q^{\,*}(z)$ from the
relation~(\ref{eq18}) by Proposition~\ref{pr6}.~$\Box$

\medskip
{\it Proof of Theorem~\ref{th1A}} directly follows by
Theorem~\ref{thOS4.2} and Lemma~\ref{l4.4}.~$\Box$

\medskip
{\it Proof of Theorem~\ref{th2}} directly follows by
Theorem~\ref{thOS4.2} and Proposition~\ref{pr6}.~$\Box$


CONTACT INFORMATION

\medskip
{\bf \noindent Evgeny Sevost'yanov} \\
{\bf 1.} Zhytomyr Ivan Franko State University,  \\
40 Bol'shaya Berdichevskaya Str., 10 008  Zhytomyr, UKRAINE \\
{\bf 2.} Institute of Applied Mathematics and Mechanics\\
of NAS of Ukraine, \\
1 Dobrovol'skogo Str., 84 100 Slavyansk,  UKRAINE\\
esevostyanov2009@gmail.com

\medskip
{\bf \noindent Valery Targonskii} \\
Zhytomyr Ivan Franko State University,  \\
40 Bol'shaya Berdichevskaya Str., 10 008  Zhytomyr, UKRAINE \\
w.targonsk@gmail.com

\end{document}